\documentclass[11pt]{amsart}
\usepackage{txfonts}
\usepackage{amsfonts}
\usepackage{amssymb}
\usepackage{amsmath}
\usepackage{mathrsfs}
\usepackage{mathtools}
\usepackage{breqn}
\usepackage[pdftex]{hyperref}

\DeclareSymbolFont{largesymbol}{OMX}{yhex}{m}{n}
\DeclareMathAccent{\Widehat}{\mathord}{largesymbol}{"62}

\title[Global Existence and Large Time Behaviors of Strong Solutions]{Global Existence and Large Time Behaviors of Strong Solutions to the Kinetic Cucker--Smale Model Coupled with the Three Dimensional  Incompressible Navier--Stokes Equations}

\author[C. Jin]{Chunyin Jin}
\address[Chunyin Jin]{\newline College of Science,\newline China Agricultural University,\newline Beijing 100083, P. R. China.}
\email{jinchunyin@163.com}

\newtheorem{theorem}{Theorem}[section]
\newtheorem{definition}{Definition}[section]
\newtheorem{lemma}{Lemma}[section]
\newtheorem{proposition}{Proposition}[section]
\newtheorem{remark}{Remark}[section]
\newcommand{\bbr}{\mathbb R}
\newcommand{\bbn}{\mathbb N}

\newcommand{\e}{\varepsilon}

\newcommand{\bx}{\mbox{\boldmath $x$}}

\newcommand{\bu}{\mbox{\boldmath $u$}}

\newcommand{\bw}{\mbox{\boldmath $w$}}
\newcommand{\bg}{\mbox{\boldmath $g$}}
\newcommand{\bh}{\mbox{\boldmath $h$}}

\begin{document}

\date{\today}



\begin{abstract}
In this paper, we study global existence and large time behaviors of strong solutions to the kinetic Cucker--Smale model coupled with the three dimensional incompressible Navier--Stokes equations in the whole space. Using the maximal regularity estimate on the Stokes equations, global-in-time strong solutions to the Cauchy problem of the coupled system are obtained under small initial data regime. The optimal decay rate of the system, in the sense that the energy of the system decays at a rate coinciding with the one corresponding to the underling incompressible Navier--Stokes equations without the coupling term, is also analyzed. Compared with previous results set in spatial-periodic or bounded domains, we circumvent the difficulty caused by unboundedness of the domain using the Fourier splitting method.
\end{abstract}

\maketitle \centerline{\date}
%%%%%%%%%%%%%%%%%%%%%%%%%%%%%%%%%%%%%%%%%%%%%%%%%%%%%%%%%%%%%%%%%%%%%%%%%%%%%%%%%%
%
%                              Sect 1. Introduction
%
%%%%%%%%%%%%%%%%%%%%%%%%%%%%%%%%%%%%%%%%%%%%%%%%%%%%%%%%%%%%%%%%%%%%%%%%%%%%%%%%%%
\section{Introduction}\label{sec-intro}
\setcounter{equation}{0}
If an ensemble of particles is immersed in the incompressible fluid, then the dynamics of the system can be modeled by the kinetic Cucker--Smale model coupled with the incompressible Navier--Stokes equations. Throughout the paper, $\nabla$ and $\Delta$ without indices denote gradient and Laplacian with respect to the spatial variable $x$, respectively. Suppose particles are transported by the fluid velocity. Then the coupled system reads as
\begin{equation} \label{eq-cs-ns}
     \begin{dcases}
         f_t + \bu \cdot \nabla_{x} f+ \nabla_{v} \cdot (L[f]f+(\bu-v)f)=0,\\
         \bu_t+ \bu \cdot \nabla \bu +\nabla P=\Delta \bu +\int_{\bbr^3}(v-\bu)fdv, \\
         \nabla \cdot \bu=0,
     \end{dcases}
\end{equation}
subject to the initial data
\begin{equation} \label{eq-sys-inidata}
  f|_{t=0}=f_0, \quad \bu|_{t=0}=\bu_0,
\end{equation}
with $\bu_0$ satisfying the compatibility condition $\nabla \cdot \bu_0=0$. Here $f(t,x, v)$ is the particle distribution function in phase space $(x, v)\in \bbr^3\times \bbr^3$ at the time $t$. $\bu$ and $P$ represent the fluid velocity and pressure, respectively. $L[f]$ is given by
\[
 L[f](t, x, v)=\int_{\bbr^{3}}\int_{\bbr^3}\varphi(|x-y|)f(t, y, v^*)(v^*-v)dy dv^*,
\]
where $\varphi(\cdot) \in C_{b}^{1}$ is a positive non-increasing function, denoting the interaction kernel. Without loss of generality, we postulate that
\[
 \max\{|\varphi|, |\varphi'|\} \le 1
\]
in the sequel. The detailed background concerning the system \eqref{eq-cs-ns} can be found in \cite{Bae2012}\cite{bae2014global}. We remark that the transport term in $\eqref{eq-cs-ns}_1$ is $\bu \cdot \nabla_{x} f$, instead of $v \cdot \nabla_{x} f$ in the usual kinetic model. This is because we suppose particles are transported by the fluid. Such transport term also appeared in previous literature \cite{Constantin2007Global}\cite{Lin2007cpam}\cite{Lin2008on}, where some micro-macro models arising from polymeric fluids are investigated.

The first equation in \eqref{eq-cs-ns} is the kinetic Cucker--Smale model with the drag term. Now let us review some backgrounds on it. In 2007, Cucker--Smale \cite{Cucker2007} introduced a system of ODEs, termed as the Cucker--Smale model, to describe flocking behaviors in multi-agent systems. Then Ha--Liu \cite{Ha2009} contributed a complete analysis on the Cucker--Smale model using the Lyapunov functional approach, and further derived the kinetic Cucker--Smale model by taking the mean-field limit to the particle model. The results on the particle model in \cite{Ha2009} were extended to measure-valued solutions to the kinetic Cucker--Smale in \cite{Carrillo2010}. See also \cite{Canizo2011} for an elegant analysis employing the optimal transport theory. As for weak and strong solutions in weighted Sobolev spaces, Jin \cite{Jin2018} recently established the well-posedness by developing a unified framework. The method of weighted energy estimates in \cite{Jin2018} is also used to deal with the kinetic equation $\eqref{eq-cs-ns}_1$ in this paper. For classical solutions to the kinetic Cucker--Smale model, we refer to \cite{Ha2008}. Taking into account stochastic factors in the modeling, there is an additional diffusive term in the kinetic Cucker--Smale model. This kind of kinetic equation is of the Fokker--Planck type, allowing for existence of an equilibrium. Duan et al. \cite{duan2010kinetic} analyzed the stability and convergence rate of classical solutions to an equilibrium under small initial perturbations in Sobolev spaces, by using the micro-macro decomposition. Now research for the Cucker--Smale model from particle descriptions to kinetic and hydrodynamic descriptions has been launched. The hydrodynamic limits to some kinetic equations were analyzed in \cite{Figalli2019}\cite{karper2015hydrodynamic}, based on the relative entropy method. For results on the hydrodynamic Cucket--Smale model, we refer to \cite{Ha2014}\cite{ha2015emergent}\cite{jin2015well}\cite{Jin}. The interested reader can also consult the review papers \cite{carrillo2010particle}\cite{choi2017emergent} for the state of the art in this research topic.

The rest two equations in \eqref{eq-cs-ns} are the three dimensional incompressible Navier--Stokes equations with the coupling term. Pioneered by the seminal work of Leray \cite{Leray1934} in 1934, there has been much extensive research on the Navier--Stokes equations. Among them, we only mention results related to this paper. In fact, the incompressible Navier--Stokes equations have an important scaling property, i.e., if $\bu(t,x)$ is a solution on the whole space, then $\bu_{\lambda}(t,x):=\lambda\bu(\lambda^2t,\lambda x)$ is also a solution. A space whose norm is invariant under this scaling is referred to as a critical space. Much work has been carried out in this setting, mainly based on two types of methods. One is by an energy type estimate, and the other by a fixed point argument. Well-posedness in $\dot{H}^{\frac12}(\bbr^3)$ is due to Fujita--Kato \cite{Fujita1964}. The existence results in the critical Besov spaces can be founded in Cannone \cite{Cannone1997A} and Planchon \cite{Planchon1998Asymptotic}. In this direction, the best result was contributed by Koch--Tataru \cite{koch2001}, where small global-in-time solutions in $BMO^{-1}$ were constructed. For the modern method in the setting of the critical $L^3(\bbr^3)$ space, we refer to the excellent book \cite{robinson2016the}. In this paper, we generalize it to the incompressible Navier--Stokes equations with the coupling term. Even though global weak solutions have been known long before in \cite{Leray1934}, their large time behaviors are investigated much later. Kato first made progress in \cite{Kato1984StrongLp}, and laid foundations for decay and existence questions for the solutions to the incompressible Navier--Stokes equations. For weak solutions with integrable initial data, Schonbek \cite{Schonbek1985L} obtained algebraic decay rate for the Cauchy problem, using the Fourier splitting method. Later, this method was extended to some other diffusive equations or with an external force. In this paper, we will generalize it to the coupled kinetic-fluid model.

Coupled kinetic-fluid models were first introduced by Williams \cite{Williams1985Combustion} in the framework of the combustion theory, and have received much attention since then, due to their applications in biotechnology, medicine, waste-water recycling, and mineral processing \cite{Carrillo2006stability}. Most previous results were concentrated on study for weak solutions in spatial-periodic domain, or for classical solutions under small smooth perturbations around an equilibrium.  Global existence and large time behavior of weak solutions to the Vlasov equation coupled with the unsteady Stokes system were investigated by Hamdache \cite{hamdache1998global}, in a bounded domain with reflection boundary conditions. Later, Boudin et al. \cite{boudin2009global} extended Hamdache's work, and studied global existence of weak solutions to the Vlasov--Navier--Stokes equations in spatial-periodic domain. Such coupled kintic-fluid equations also appeared in the setting of flocking particles immersed in fluids. For the kinetic Cucker--Smale model coupled with the Stokes equations, incompressible Navier--Stokes equations, and isentropic compressible Navier--Stokes equations, existence and large time behaviors of weak or strong solutions in spatial-periodic domain have been analyzed in  \cite{Bae2012}\cite{bae2014asymptotic}\cite{bae2014global}\cite{bae2016global}. We refer the reader to \cite{carrillo2011global}\cite{goudon2010navier}\cite{li2017strong} for global classical solutions to the coupled kinetic-fluid models under small perturbations regime.

Recently, the author has started the program to study the Cauchy problem of the kinetic Cucker--Smale model coupled with fluid equations. Local existence and a blowup criterion of strong solutions to the kinetic Cucker--Smale model coupled with the isentropic compressible Navier--Stokes equations were obtained in Jin \cite{jin2019local}, where weighted Sobolev spaces were introduced to overcome difficulties arising from unboundedness of the domain and the coupling term. It was shown that the integrability in time of the spatial $W^{1,\infty}$-norm of $\bu(t,x)$ controls blowup of strong solutions. Along this line, Jin \cite{Jin2021} investigated global existence of strong solutions to the kinetic Cucker--Smale model coupled with the Stokes equations, by using the maximal regularity estimate on the Stokes equations. In this paper, we further study the kinetic Cucker-Smale model coupled with the three dimensional incompressible Navier--Stokes equations. It is remarked that there are two aspects different from most previous results: First, the regularity of initial data is minimal in terms of existence of strong solutions to the incompressible Navier--Stokes equations. We control the quantity determining the extendence of local strong solutions by invoking the maximal regularity estimate on the Stokes equations, instead of using high order energy estimates, which requires initial data lying in high order Sobolev spaces; second, large time behaviors of strong solutions are analyzed in the whole space, instead of bounded or spatial-periodic domain. Thus the Poincar{\'e} inequality is not applicable. We circumvent this difficulty by ingeniously using the Fourier splitting method.

Motivated by our previous studies \cite{jin2019local}\cite{Jin2021}, we introduce the following weighted Sobolev space to overcome difficulty induced from unboundedness of the domain.
\begin{multline*}
 H_{\omega}^1(\bbr^3 \times \bbr^3):=\bigg\{h(x,v):\ h \in L_{\omega}^2(\bbr^3 \times \bbr^3), \\ \nabla_{x} h\in L_{\omega}^2(\bbr^3 \times \bbr^3),\ \nabla_{v}h \in L_{\omega}^2(\bbr^3 \times \bbr^3)\bigg\},
\end{multline*}
\[
 \|h\|_{H_{\omega}^1}^2:=\|h\|_{L_{\omega}^2}^2+\|\nabla_{x}h\|_{L_{\omega}^2}^2+\|\nabla_{v}h\|_{L_{\omega}^2}^2,
\]
where
\[
 \|h\|_{L_{\omega}^2}:=\left(\int_{\bbr^3}\int_{\bbr^3}h^2(x,v)\omega(x)dx dv\right)^{\frac12}, \quad \omega(x):=(1+|x|^2)^{\alpha},\quad \alpha>\frac32,
\]
similarly for $\|\nabla_{x}h\|_{L_{\omega}^2}$ and $\|\nabla_{v}h\|_{L_{\omega}^2}$. The following simplified notations for homogeneous Sobolev Spaces are used in this paper.
\[
 D^1(\bbr^3):=\left\{u\in L^{6}(\bbr^3):\ \nabla u \in L^2(\bbr^3)\right\},
\]
\[
 D^2(\bbr^3):=\left\{u\in L_{loc}^1(\bbr^3):\ \nabla^2u \in L^2(\bbr^3)\right\},
\]
\[
 D^{2,p}(\bbr^3):=\left\{u\in L_{loc}^1(\bbr^3):\ \nabla^2u \in L^p(\bbr^3)\right\}, \quad 1\le p\le \infty.
\]
Next we give the definition of strong solutions to \eqref{eq-cs-ns}-\eqref{eq-sys-inidata}.
\begin{definition} \label{def-stro}
Let $0<T< \infty$. $(f(t,x,v), \bu(t,x),\nabla P(t,x))$ is said to be a strong solution to \eqref{eq-cs-ns}-\eqref{eq-sys-inidata} in $[0,T]$, if
\[
\begin{aligned}
  &f(t,x,v)\in C([0,T];H_{\omega}^1(\bbr^3 \times \bbr^3)),\\
  &\bu(t,x)\in C([0,T];H^{1}(\bbr^3))\cap L^2(0,T;D^{2}(\bbr^3)),\\
  &\bu_t(t,x)\in L^2(0,T;L^{2}(\bbr^3)),\\
  &\nabla P(t,x)\in L^2(0,T;L^{2}(\bbr^3)),
\end{aligned}
\]
and
\[
 \begin{gathered}
   \int_0^{\infty}\int_{\bbr^6}f \phi_t dx dv dt+\int_0^{\infty}\int_{\bbr^6}f v \cdot \nabla_{x}\phi dx dv dt\\+\int_0^{\infty}\int_{\bbr^6}\Big(L[f]+\bu-v \Big)\cdot \nabla_{v}\phi f dx dv dt
   +\int_{\bbr^6}f_0 \phi(0) dx dv=0,
 \end{gathered}
\]
for all $\phi(t,x,v)\in C_0^{\infty}([0,T)\times \bbr^3 \times \bbr^3)$;
\[
 \begin{gathered}
   \int_0^{\infty}\int_{\bbr^3}\bu\cdot \boldsymbol{\psi}_t dx dt+\int_0^{\infty}\int_{\bbr^3}\nabla \bu : \nabla \boldsymbol{\psi} dx dt -\int_0^{\infty}\int_{\bbr^3} \bu \cdot \nabla \bu \cdot \boldsymbol{\psi} dx dt
   \\+\int_0^{\infty}\int_{\bbr^6}(v-\bu) \cdot \boldsymbol{\psi}f dx dv dt +\int_{\bbr^3}\bu_0 \cdot\boldsymbol{\psi}(0) dx=0,
 \end{gathered}
\]
for all $\boldsymbol{\psi}(t,x)\in C_{0,\sigma}^{\infty}([0,T)\times \bbr^3)$, where
\[
 C_{0,\sigma}^{\infty}([0,T)\times \bbr^3):=\bigg\{\boldsymbol{\psi}(t,x): \ \boldsymbol{\psi}(t,x)\in C_0^{\infty}([0,T)\times \bbr^3), \ \nabla \cdot \boldsymbol{\psi}=0 \bigg\}.
\]
\end{definition}
We define the particle density as
\[
 \rho(t,x):=\int_{\bbr^3}f(t,x,v)dv,
\]
with the initial value $\rho_0(x):=\rho(0,x)$. The energy of the system \eqref{eq-cs-ns} is defined as
\[
 E(t):=\frac12 \|\bu(t)\|_{L^2}^2 +\frac12 \int_{\bbr^3}\int_{\bbr^3}f(t,x,v)|v|^2dxdv,
\]
with the initial energy $E_0:=E(0)$. If the initial fluid velocity $\bu_0(x)\in H^1(\bbr^3)$, it follows from the Sobolev embedding that
\[
 \|\bu_0\|_{L^3}\le C \|\bu_0\|_{H^1},\quad \text{for some constant $C>0$.}
\]
Denote by $B(R_0)$ the ball centered at the origin with a radius $R_0$. The bound of $v$-support of $f(t,x, v)$ at the time $t$ is defined as
\[
  R(t):=\sup \Big\{|v|: \ v \in \text{supp$f(t,x,\cdot)$ for a.e. $x \in \bbr^3$}\Big \}.
\]
If $f_0(x,v) \in H_{\omega}^1(\bbr^3 \times \bbr^3)\cap L^{\infty}(\bbr^3 \times \bbr^3)$ and $\text{supp}_{v}f_0(x,\cdot)\subseteq B(R_0)$ for a.e. $x \in \bbr^3$, then we have
\[
  \|\rho_0\|_{L^1}=\int_{\bbr^3}\int_{\bbr^3}f_0(x,v)dxdv\le C R_0^{\frac32} \|f_0\|_{L_{\omega}^2},
\]
and
\[
  \|\rho_0\|_{L^{\infty}}=\left\|\int_{\bbr^3}f_0(x,v)dv \right\|_{L^{\infty}}\le C R_0^3 \|f_0\|_{L^{\infty}},
\]
for some constant $C>0$. In terms of the above notations, the results of this paper are summarized as follows.
\begin{theorem} \label{thm-exist}
Let $0<R_0<\infty$. Assume that $f_0(x,v)\ge 0$, $f_0(x,v) \in H_{\omega}^1(\bbr^3 \times \bbr^3)\cap L^{\infty}(\bbr^3 \times \bbr^3)$, and $\bu_0(x) \in H^{1}(\bbr^3)$,
with $\nabla \cdot \bu_0=0$ and the $v$-support of $f_0(x,v)$ satisfying
\[
 \text{supp}_{v}f_0(x,\cdot)\subseteq B(R_0) \quad \text{for a.e. $x \in \bbr^3$}.
\]
If there is a sufficiently small constant $\varepsilon_0$ such that
\[
  C\|\bu_0\|_{L^3}^6+(1+\|\rho_0\|_{L^{\infty}})E_0 \le \varepsilon_0
\]
for some constant $C>0$ independent of the initial data, then the Cauchy problem \eqref{eq-cs-ns}-\eqref{eq-sys-inidata} admits a unique global-in-time strong solution in the sense of Definition \ref{def-stro}. Moreover, it holds that
\[
 \begin{aligned}
   &(1)\ \|f(t)\|_{H_{\omega}^1}\le \|f_0\|_{H_{\omega}^1}\exp\Big(C(1+t)\Big) \quad \text{and} \quad R(t)\le R_0+C,\\&\quad\, \text{where $C:=C(R_0, E_0, \|\rho_0\|_{L^{\infty}}, \|\bu_0\|_{H^1},\|\bu_0\|_{L^3})$};\\
   &(2)\ \|\bu(t)\|_{H^{1}}\le C\bigg(\|\nabla \bu_0\|_{L^2}+\Big(1+\|\rho_0\|_{L^{\infty}}^{\frac12}\Big)E_0^{\frac12}\bigg)\exp\left(C(\varepsilon_0+\varepsilon_0^{\frac12})\right),\\
   &\quad\, \text{for some constant $C>0$ independent of the initial data};\\
   &(3)\ \lim_{t \to \infty}\int_{\bbr^6} |v-\bu|^2 f dxdv =0.
 \end{aligned}
\]
\end{theorem}
If $\bu_0(x)$ is integrable, except for the conditions in Theorem \ref{thm-exist}, then the quantitative decay rate of the system can be obtained.
\begin{theorem} \label{thm-beha}
Assume initial data satisfy the conditions in Theorem \ref{thm-exist}. If $\bu_0(x)\in L^1(\bbr^3)$ in addition, then there holds
\[
 \begin{aligned}
   &(1)\ E(t)\le C\Big(1+t\Big)^{-\frac32} \quad \text{and} \quad \|\bu(t)\|_{L^{2}}\le C\Big(1+t\Big)^{-\frac34};\\
   &(2)\ \int_{\bbr^6} |v-\bu|^2 f dxdv \le C\Big(1+t\Big)^{-\frac32},
 \end{aligned}
\]
where $C:=C(E_0, \|\rho_0\|_{L^\infty}, \|\bu_0\|_{L^1})$.
\end{theorem}
\begin{remark}\label{rem-thm}
The transport term in $\eqref{eq-cs-ns}_1$ is taken as $\bu \cdot \nabla_xf$, in order to derive a uniform-in-time bound on the particle density $\rho(t,x)$, which plays a crucial role in the analysis of global existence and large time behaviors of strong solutions. As for the usual transport term $v \cdot \nabla_xf$, particle concentration in phase space seems unavoidable. Whether Theorem \ref{thm-exist} holds or not in this setting is still unknown, due to lack of a uniform-in-time bound on the particle density.
\end{remark}
\begin{remark}
The regularity of $\bu_0$ is optimal in terms of existence of strong solutions to the incompressible Navier--Stokes equations. The small global strong solution is obtained in the critical $L^3(\bbr^3)$ space, based on an energy-type estimate. As is well-known, the best critical space to date for existence of small global solutions to the three dimensional incompressible Navier--Stokes equations is $BMO^{-1}$. Whether or not the same type results hold for small $\bu_0$ in $BMO^{-1}$ is still a problem deserving our further endeavor.
\end{remark}
\begin{remark}
Compared with decay results in the three dimensional incompressible Navier--Stokes equations, the decay rates in Theorem \ref{thm-beha} are optimal. Since the smallness assumption is only used to derive global strong solutions, which does not play a role in the analysis of large time behaviors, we conjecture that Theorem \ref{thm-beha} also holds for large weak solutions to the system \eqref{eq-cs-ns}-\eqref{eq-sys-inidata}.
\end{remark}

Based on our analysis on local strong solutions to the coupled system, it suffices to prove $\int_0^T \|\bu(t)\|_{W^{1,\infty}} dt<\infty$ for all $T>0$, to extend local solutions to infinity. This estimate can be obtained by interpolation, once we have the estimate on $\|\nabla^2\bu\|_{L^s(0,T;L^q)}$, $1<s<2$, $3<q<6$.
In fact, the estimate of $\|\nabla^2\bu\|_{L^s(0,T;L^q)}$ can be obtained, by employing the maximal regularity estimate on the Stokes equations. Nevertheless, the incompressible Navier--Stokes equations contain the convective term $\bu\cdot\nabla\bu$. To deal with the convective term, we need  the estimate on $\sup_{0\le t\le T}\|\nabla\bu(t)\|_{L^2}$. However, the estimate on $\sup_{0\le t\le T}\|\nabla\bu(t)\|_{L^2}$ in the three dimensional incompressible Navier--Stokes equations is still open for large initial data. Thus this estimate by far can only be hoped for under small initial data regime. In this paper, we generalize the small global existence result in the critical $L^3(\bbr^3)$ space for the incompressible Navier--Stokes equations to the coupled system \eqref{eq-cs-ns}-\eqref{eq-sys-inidata}. Using the splitting approach, we decompose $\bu=\bar\bu +\bw$, with $\bar\bu$ and $\bw$ determined by the following equations:
\begin{equation*} \label{eq-ns-baru}
     \begin{dcases}
         \bar\bu_t+ \bu \cdot \nabla \bar\bu =\Delta \bar\bu, \\
         \bar\bu|_{t=0}=\bu_0,
     \end{dcases}
\end{equation*}
and
\begin{equation*} \label{eq-ns-w}
     \begin{dcases}
         \bw_t+ \bu \cdot \nabla \bw +\nabla P=\Delta \bw +\int_{\bbr^3}(v-\bu)fdv, \\
         \bw|_{t=0}=0,
     \end{dcases}
\end{equation*}
respectively. In the process of the $L^3$-type energy estimate, the estimate on $\bar\bu$ is standard under the incompressible condition $\nabla\cdot\bu=0$. The difficulty mainly arises from the term $\int_{\bbr^3}3|\bw|\bw\cdot\nabla P dx$, where the pressure $P$ is determined by the following equation:
\[
 \Delta P=-\nabla\cdot(\bu\cdot\nabla\bu)+\nabla\cdot\bh, \quad \text{with $\bh:=\int_{\bbr^3}(v-\bu)f dv$.}
\]
Appearance of the additional term $\nabla\cdot\bh$ disables us to use integration by parts to estimate $\int_{\bbr^3}3|\bw|\bw\cdot\nabla P dx$. We overcome this difficulty by decomposing $P=P_1+P_2$, with $P_1$ and $P_2$ governed by
\[
 \Delta P_1=-\nabla\cdot(\bu\cdot\nabla\bu)\quad\text{and}\quad \Delta P_2=\nabla\cdot\bh,
\]
respectively. Then it holds that
\[
 \int_{\bbr^3}3|\bw|\bw\cdot\nabla P dx=\int_{\bbr^3}3|\bw|\bw\cdot\nabla P_1 dx+\int_{\bbr^3}3|\bw|\bw\cdot\nabla P_2 dx.
\]
The first term in the right-hand side of the above equation is dealt with similarly as in the incompressible Navier--Stokes equations, i.e., using integration by parts and the inequality
\[
 \|P_1\|_{L^p}\le C\|\bu\|_{L^{2p}}^{2}\quad \text{for some constant $C>0$ and $1<p<\infty$}.
\]
The second term is estimated directly, using H\"older's inequality and the fact that
\[
 \|\nabla P_2\|_{L^p}\le C\|\bh\|_{L^{p}}\quad \text{for some constant $C>0$ and $1<p<\infty$}.
\]
Having the uniform-in-time bound on the particle density $\rho(t,x)$, along with the elementary energy estimate on the system \eqref{eq-cs-ns}-\eqref{eq-sys-inidata}, the quantity $\int_0^T \|\bh(t)\|_{L^p}^2dt$ for $1\le p\le 2$ can be controlled by the initial energy $E_0$ times $1+\|\rho_0\|_{L^{\infty}}$ for all $T>0$. Then we can obtain the Serrin condition $\bu\in L^3(0,T;L^9)$ for all $T>0$, under some smallness assumption on initial data, which leads to estimate on $\sup_{0\le t\le T}\|\nabla\bu(t)\|_{L^2}$ for all $T>0$.

In analysis of large time behaviors to the coupled system, we use the Fourier splitting method. Nevertheless, we cannot directly apply this method to the fluid equations, since
\[
 \frac12\frac{d}{dt}\|\bu\|_{L^2}^2 +\|\nabla\bu\|_{L^2}^2=\int_{\bbr^6}(v-\bu)\cdot\bu f dxdv,
\]
and we do not have enough decay estimates on the coupling term $\int_{\bbr^6}(v-\bu)\cdot\bu f dxdv$. Combining the kinetic equation, we deduce that
\[
 \frac{d}{dt}E(t) +\|\nabla\bu\|_{L^2}^2+\int_{\bbr^6}|v-\bu|^2 f dxdv \le 0.
\]
Our main observation is that
\[
  E(t)\le (1+\|\rho_0\|_{L^{\infty}})\left(\|\bu\|_{L^2}^2+\int_{\bbr^6}|v-\bu|^2 f dxdv\right),
\]
under uniform-in-time bound on the particle density $\rho(t,x)$. Then we use the Fourier splitting method to derive the inequality
\[
 \frac{d}{dt}E(t) +\frac{c^2}{(t+c^2)(1+\|\rho_0\|_{L^{\infty}})}E(t) \le \frac{c^2}{t+c^2}\int_{|\xi|\le c(t+c^2)^{-\frac12}} |\hat\bu(t,\xi)|^2d\xi,
\]
where the positive constant $c$ satisfies $c^2=3(1+\|\rho_0\|_{L^{\infty}})$. However, due to appearance of the coupling term, we are still faced with some difficulties in estimating the low frequency of $\bu(t,x)$. Fortunately, these difficulties can be overcome by a subtle bootstrap argument. It follows from the elementary energy estimate that $E(t)\le E_0$ for all $t\ge 0$. Therefore, we can make the bootstrap assumption that
\[
 E(t)<K(1+t)^{-\frac98} \quad \text{in $[0,T)$}
\]
for some suitably large constants $K$ and $T$. Under this assumption, the coupling term can be controlled by using a time-weighted energy estimate on the coupled system
\eqref{eq-cs-ns}-\eqref{eq-sys-inidata}. Then we can improve estimate on the low frequency of $\bu(t,x)$, which leads to a better decay rate on $E(t)$, i.e.,
\[
 E(t)\le CK(1+t)^{-\frac32}+CK^2(1+t)^{-\frac52} \quad \text{in $[0,T)$}
\]
for some constant $C>0$, depending only on $E_0$, $\|\rho_0\|_{L^\infty}$ and $\|\bu_0\|_{L^1}$; see details in Sect. \ref{sec-glob-exst}. Taking $K$ suitably large, we can prove that the supremum among all the above $T$ is $\infty$ by continuity argument. Thus we can improve the bootstrap assumption and show that
\[
 E(t)\le C(1+t)^{-\frac32} \quad \text{for all $t\ge 0$},
\]
where $C:=C(E_0, \|\rho_0\|_{L^\infty}, \|\bu_0\|_{L^1})$. Then the decay rates of $\|\bu(t)\|_{L^2}$ and $\int_{\bbr^6} |v-\bu|^2 f dxdv$ follow easily.
The proof is subtle and robust. We believe that the idea can be also used in some other coupled models.

The rest of the paper is organized as follows. In Sect. \ref{sec-preli}, we present some preliminary results used in the subsequent analysis. In Sect. \ref{set-loc-exist}, we construct local strong solutions to the coupled model by iteration on the linearized system. In Sect. \ref{sec-apriori}, we derive some a priori estimates on classical solutions, which are used to extend local strong solutions. Sect. \ref{sec-glob-exst} is devoted to the proof of our theorems.
\vskip 0.3cm
\noindent\textbf{Notation}. Throughout the paper, $c, C$  represent a general positive constant that may depend on $\varphi$, $\varphi'$, and the initial data. We write $C(\star)$ to emphasize that $C$ depends on $\star$. Both $C$ and $C(\star)$ may take different values in different expressions.

%%%%%%%%%%%%%%%%%%%%%%%%%%%%%%%%%%%%%%%%%%%%%%%%%%%%%%%%%%%%%%%%%%%%%%%%%%%%%%%%%%%%%%%%%%%%%%%%%%%%%%%%%%%%%%%
%
%                     Sect. 2  Preliminary
%
%%%%%%%%%%%%%%%%%%%%%%%%%%%%%%%%%%%%%%%%%%%%%%%%%%%%%%%%%%%%%%%%%%%%%%%%%%%%%%%%%%%%%%%%%%%%%%%%%%%%%%%%%%%%%%%
\section{Preliminaries}\label{sec-preli}
\setcounter{equation}{0}
In this section, we present some preliminaries that will be used in the following analysis. The first result concerns the well-posedness and estimates on strong solutions to the kinetic Cucker--Smale model with a coupling term, provided that the fluid velocity is given in some function space. It will be used in the construction of local strong solutions to the system \eqref{eq-cs-ns}-\eqref{eq-sys-inidata}. While the second result is the maximal regularity estimate on the Stokes equations. This estimate is crucially used in the estimate of $\int_0^T \|\bu(t)\|_{W^{1,\infty}} dt<\infty$ for all $T>0$, in order to extend local strong solutions to infinity.

\subsection{The kinetic Cucker--Smale model with a coupling term}
When an ensemble of particles is immersed in a fluid, due to influence of ambient fluid, the alignment term $L[f]f$ is  replaced by $L[f]f+(\bu-v)f$, in order for conservation of mass and momentum of the system. In this paper, we suppose particles are transported by the fluid velocity $\bu$. Then the transport term in the kinetic equation becomes $\bu\cdot\nabla_xf$. Given $\bu(t,x)\in C([0,T];H^{1}(\bbr^3))\cap L^2(0,T;D^{2}(\bbr^3)) \cap L^s(0,T;D^{2,q}(\bbr^3))$,$1<s<2$, $3<q<6$, with $\nabla\cdot\bu=0$, it follows from the Sobolev imbedding that $\bu(t,x)\in L^1(0,T;W^{1,\infty}(\bbr^3))$.
Consider
\begin{equation} \label{eq-kine-cs}
     \begin{dcases}
         f_t + \bu \cdot \nabla_{x} f+ \nabla_{v} \cdot (L[f]f+(\bu-v)f)=0,\\
         f|_{t=0}=f_0(x,v),
     \end{dcases}
\end{equation}
in $[0,T]\times \bbr^3 \times \bbr^3$. Define
\[
 a(t,x):=\int_{\bbr^{6}} \varphi(|x-y|)f(t, y,v^*) dy dv^*,
\]
\[
 \mathbf{b} (t,x):=\int_{\bbr^{6}} \varphi(|x-y|)f(t, y,v^*) v^* dy dv^*.
\]
Recall that the bound of $v$-support of $f(t,x, v)$ at the time $t$ is defined as
\[
  R(t):=\sup \Big\{|v|: \ v \in \text{supp$f(t,x,\cdot)$ for a.e. $x \in \bbr^3$}\Big \}.
\]
Following the line of proof of Proposition 2.1 in \cite{jin2019local}, we have the following results.
\begin{proposition}\label{prop-kine-cs-wp}
Let $0<R_0, T<\infty$. Assume $f_0(x,v) \ge 0$, $f_0(x,v) \in H_{\omega}^1(\bbr^3 \times \bbr^3)$, and $\text{supp}_{v}f_0(x,\cdot)\subseteq B(R_0)$ for a.e. $x \in \bbr^3$. Given $\bu(t,x)\in C([0,T];H^{1}(\bbr^3))\cap L^2(0,T;D^{2}(\bbr^3)) \cap L^s(0,T;D^{2,q}(\bbr^3))$,$1<s<2$, $3<q<6$, with $\nabla\cdot\bu=0$, there exists a unique non-negative strong solution $f(t, x, v)\in C([0,T];H_{\omega}^1(\bbr^3 \times \bbr^3))$ to \eqref{eq-kine-cs}. Moreover,
\[
 \begin{aligned}
   (1)& \ R(t)\le R_0+\int_0^t(\|\mathbf{b}(\tau)\|_{L^{\infty}}+\|\bu(\tau)\|_{L^{\infty}})d\tau,\quad 0\le t \le T ;\\
   (2)& \ \|f(t)\|_{H_{\omega}^1}\le \|f_0\|_{H_{\omega}^1}\exp \left( C \int_0^t \Big(1+R(\tau)+\|\bu(\tau)\|_{W^{1,\infty}} \Big) d \tau \right), \quad 0\le t \le T,
 \end{aligned}
\]
where $C:=C(\|\varphi\|_{C^1}, \|f_0\|_{L^1})$.
\end{proposition}
From $f_0(x,v) \in H_{\omega}^1(\bbr^3 \times \bbr^3)$, we deduce that
\begin{equation} \label{eq-ini-f-lonenm}
 \|f_0\|_{L^1}=\int_{\bbr^6}f_0(x,v)\omega^{\frac12}(x)\omega^{-\frac12}(x)dx dv \le C(R_0)\|f_0\|_{L^2_{\omega}}.
\end{equation}
Integrating $\eqref{eq-kine-cs}_1$ over $[0,t]\times \bbr^3 \times \bbr^3$ $, 0<t\le T$, gives
\begin{equation} \label{eq-kin-conser-mass}
 \|f(t)\|_{L^1}=\|f_0\|_{L^1}.
\end{equation}
Without loss of generality, we suppose $\|f_0\|_{L^1}=1$. Using Cauchy's inequality, we have
\begin{equation}\label{eq-estb}
 \begin{aligned}
   |\mathbf{b}(t,x)| \le&  \left(\int_{\bbr^{6}} f(t, y,v^*) dy dv^* \right)^{\frac12}\left(\int_{\bbr^{6}} f(t, y, v^*) |v^*|^2 dy dv^* \right)^{\frac12}\\
   \le& \left(\int_{\bbr^{6}} f(t, y, v^*) |v^*|^2 dy dv^* \right)^{\frac12}, \quad 0\le t \le T.
 \end{aligned}
\end{equation}
Multiplying $\eqref{eq-kine-cs}_1$ by $|v|^2$, we obtain
\begin{equation} \label{eq-kin-cs-ener}
  \frac{\partial}{\partial t} (f |v|^2)+ \bu \cdot \nabla_{x}(f |v|^2)+\nabla_{v} \cdot \Big(L[f]f |v|^2+(\bu-v)f |v|^2\Big)=2fL[f]\cdot v+2f v \cdot (\bu-v).
\end{equation}
Integrating \eqref{eq-kin-cs-ener} over $\bbr^3 \times \bbr^3$ leads to
\[
 \begin{aligned}
   \frac{d}{dt}\int_{\bbr^6}f|v|^2 dx dv=&-\int_{\bbr^{12}}\varphi(|x-y|)f(t,x,v)f(t,y,v^*)|v^*-v|^2dy dv^* dx dv\\
     &-2\int_{\bbr^6}f|v|^2 dx dv+2\int_{\bbr^3}\int_{\bbr^3}f v dv \cdot \bu dx\\
     \le&-2\int_{\bbr^6}f|v|^2 dx dv+2\|\bu(t)\|_{L^{\infty}}\|f\|_{L^1}^{\frac12}\Bigg(\int_{\bbr^6}f|v|^2 dx dv\Bigg)^{\frac12}.
 \end{aligned}
\]
Solving the above Gronwall's inequality yields
\begin{equation}\label{eq-cs-twmonieq}
 \Bigg(\int_{\bbr^6}f(t, x, v)|v|^2 dx dv\Bigg)^{\frac12}\le R_0+\int_0^t \|\bu(\tau)\|_{L^{\infty}}d\tau, \quad 0\le t \le T.
\end{equation}
Combining \eqref{eq-estb} and \eqref{eq-cs-twmonieq} with Proposition \ref{prop-kine-cs-wp} (1), we can further estimate $R(t)$ as follows:
\begin{equation}\label{eq-cs-estvbd}
 R(t)\le \left(R_0+\int_0^t \|\bu(\tau)\|_{L^{\infty}}d\tau\right)(1+t),\quad 0\le t \le T.
\end{equation}

\subsection{Maximal regularity estimate on the Stokes equations}
Our analysis to the incompressible Navier--Stokes equations is based on the global-in-time maximal $L^p-L^q$ estimates for the Stokes equations. By global-in-time maximal $L^p-L^q$ estimates, we mean that for given $\bg(t,x)\in L^p(0,T;L^q(\bbr^3))$, there exists a unique solution
$(\bu,\nabla P)$ to the Stokes equations
\begin{equation} \label{eq-s}
     \begin{dcases}
         \bu_t+\nabla P=\Delta \bu +\bg, \\
         \nabla \cdot \bu=0,\\
         \bu|_{t=0}=0,
     \end{dcases}
\end{equation}
satisfying
\[
 \|\bu_t\|_{L^p(0,T;L^q)}+\|\nabla^2\bu\|_{L^p(0,T;L^q)}+\|\nabla P\|_{L^p(0,T;L^q)}\le C\|\bg\|_{L^p(0,T;L^q)},
\]
for some constant $C>0$ independent of $T$ and $\bg$. Pioneering results on maximal $L^p-L^p$ estimates are due to Solonnikov \cite{Solonnikov1977Estimates}. Later, by proving boundedness of the imaginary powers of the Stokes operator, and using the Dore--Venni Theorem \cite{Dore1987On}, Giga--Sohr \cite{giga1991abstract} extended these estimates to the mixed $L^p-L^q$ setting, with the above constant $C$ improved to be independent of $T$, i.e., global in time. For modern approach to the maximal regularity estimates from the point of view of evolution equations, we refer the reader to \cite{Pruss2016Moving}. In terms of the regularity of $\bu(t,x)$, the space of initial data $\bu_0(x)$, as the trace space of $\bu(t,x)$, can be characterized by the real interpolation space
\[
  \Big(W^{2,q}, L^q\Big)_{\frac1s,s}=B_{q,s}^{2-\frac2s}.
\]
For general initial data in the Besov space $B_{q,s}^{2-\frac2s}$, it is proved in \cite{giga1991abstract} that
\[
 \|\bu_t\|_{L^s(0,T;L^q)}+\|\nabla^2\bu\|_{L^s(0,T;L^q)}+\|\nabla P\|_{L^s(0,T;L^q)}\le C\bigg(\|\bu_0\|_{B_{q,s}^{2-\frac2s}}+\|\bg\|_{L^s(0,T;L^q)}\bigg).
\]
In this paper, we suppose
\begin{equation} \label{eq-sq-rela}
 1<s<2, \quad 3<q<6, \quad \text{and} \quad \frac2s+\frac3q>\frac52.
\end{equation}
It follows from the Sobolev embedding that
\begin{equation} \label{eq-s-inidaemb}
 \|\bu_0\|_{B_{q,s}^{2-\frac2s}}\le C\|\bu_0\|_{H^1}, \quad \text{for some constant $C>0$.}
\end{equation}
Combining Theorem 2.8 in \cite{giga1991abstract} with \eqref{eq-s-inidaemb}, we have the following proposition.

\begin{proposition}\label{prop-s}
Let $0<T\le\infty$, $1<s<2$, $3<q<6$, $\frac2s+\frac3q>\frac52$. Given $\bg(t,x)\in L^s(0,T;L^q(\bbr^3)$, $\bu_0(\bx)\in H^1(\bbr^3)$, there exists a unique solution
$(\bu,\nabla P)$ to the Stokes equations
\begin{equation} \label{eq-s-nzroini}
     \begin{dcases}
         \bu_t+\nabla P=\Delta \bu +\bg, \\
         \nabla \cdot \bu=0,\\
         \bu|_{t=0}=\bu_0,
     \end{dcases}
\end{equation}
such that for some constant $C>0$ independent of $T$ and $\bg$,
\[
 \|\bu_t\|_{L^s(0,T;L^q)}+\|\nabla^2\bu\|_{L^s(0,T;L^q)}+\|\nabla P\|_{L^s(0,T;L^q)}\le C\Big(\|\bu_0\|_{H^1}+\|\bg\|_{L^s(0,T;L^q)}\Big).
\]
If $\bg(t,x)\in L^2(0,T;L^2(\bbr^3)$, it holds similarly that
\[
 \|\bu_t\|_{L^2(0,T;L^2)}+\|\nabla^2\bu\|_{L^2(0,T;L^2)}+\|\nabla P\|_{L^2(0,T;L^2)}\le C\Big(\|\bu_0\|_{H^1}+\|\bg\|_{L^2(0,T;L^2)}\Big),
\]
for some constant $C>0$ independent of $T$ and $\bg$.
\end{proposition}
%%%%%%%%%%%%%%%%%%%%%%%%%%%%%%%%%%%%%%%%%%%%%%%%%%%%%%%%%%%%%%%%%%%%%%%%%%%%%%%%%%%%%%%%%%%%%%%%%%%%%%%%%%%%%%%%
%
%                   Sect.3 Local Existence of Strong Solutions to the Coupled System
%
%%%%%%%%%%%%%%%%%%%%%%%%%%%%%%%%%%%%%%%%%%%%%%%%%%%%%%%%%%%%%%%%%%%%%%%%%%%%%%%%%%%%%%%%%%%%%%%%%%%%%%%%%%%%%%%%
\section{Construction of Local Strong Solutions}\label{set-loc-exist}
\setcounter{equation}{0}
This section is devoted to the construction of local strong solutions to the coupled system. Based on our previous results for the kinetic Cucker--Smale model, we first linearize the incompressible Navier--Stokes equations and construct approximate solutions by iteration. Then we show that there exists some $T_*>0$, depending only on initial data and model parameters, such that the approximate solutions are uniformly bounded in strong solution spaces on $[0, T_*]$. Then it is proved that the approximate solution sequence is a Cauchy sequence in a lower order regularity function space for some short time interval. Combining the uniform bound on the approximate solutions, we further demonstrate that the limit is the desired local strong solution, by employing functional analysis. The result in this section is summarized as follows.
\begin{proposition} \label{prop-loc-exist}
Let $0<R_0<\infty$. Assume that $f_0(x, v)\ge 0$, $f_0(x, v) \in H_{\omega}^1(\bbr^3 \times \bbr^3)$, and $\bu_0(x) \in H^{1}(\bbr^3)$,
with $\nabla \cdot \bu_0=0$ and the $v$-support of $f_0(x, v)$ satisfying
\[
 \text{supp}_{v}f_0(x,\cdot)\subseteq B(R_0) \quad \text{for a.e. $x \in \bbr^3$}.
\]
Then there exists some $T_0>0$, depending only on $R_0$, $\|\bu_0\|_{H^1}$, $\|f_0\|_{H_{\omega}^1}$, such that the Cauchy problem \eqref{eq-cs-ns}-\eqref{eq-sys-inidata} admits a unique strong solution on $[0,T_0]$, satisfying
\[
\begin{aligned}
  (1) \ &\sup_{0\le t\le T_0}\|f(t)\|_{H_{\omega}^1}\le 2\|f_0\|_{H_{\omega}^1}, \ \text{and $R(t)\le 2R_0$ for $t\in [0, T_0]$};\\
  (2)\ &\sup_{0\le t\le T_0}\|\bu(t)\|_{H^1}+\|\bu_t\|_{L^2(0,T_0;L^2)}+\|\nabla^2\bu\|_{L^2(0,T_0;L^2)}+ \|\nabla^2\bu\|_{L^s(0,T_0;L^q)}\le 2K_0 \|\bu_0\|_{H^1}\\
  &\text{for some constant $K_0>0$.}
\end{aligned}
\]
\end{proposition}
Next we use results in Sect. \ref{sec-preli} to complete the proof of Proposition \ref{prop-loc-exist}.
\vskip 0.3cm
\noindent \textit{Proof of Proposition \ref{prop-loc-exist}}. We first construct approximate solutions by linearizing the incompressible Navier--Stokes equations. Given $\bu^n(t, x) \in C([0,T];H^{1}(\bbr^3))\cap L^2(0,T;D^{2}(\bbr^3))\cap L^s(0,T;D^{2,q}(\bbr^3))$, $T>0$, with $\nabla \cdot \bu^{n}=0$  and $\bu^n|_{t=0}=\bu_0$, $(f^{n+1},\bu^{n+1},\nabla P^{n+1})$ is determined by
\begin{equation} \label{eq-cs-s-appro}
     \begin{dcases}
         f^{n+1}_t + \bu^n \cdot \nabla_{x} f^{n+1}+ \nabla_{v} \cdot (L[f^{n+1}]f^{n+1}+(\bu^{n}-v)f^{n+1})=0,\\
         \bu^{n+1}_t+\bu^n \cdot \nabla \bu^{n+1}+ \nabla P^{n+1}=\Delta \bu^{n+1}+\int_{\bbr^3}(v-\bu^{n+1})f^{n+1}dv,\\
         \nabla \cdot \bu^{n+1}=0,
     \end{dcases}
\end{equation}
subject to the initial data
\begin{equation} \label{eq-cs-s-approcau}
 f^{n+1}|_{t=0}=f_0,  \quad \bu^{n+1}|_{t=0}=\bu_0,
\end{equation}
with $\bu_0$ satisfying the compatibility condition $\nabla \cdot \bu_0=0$. From Proposition \ref{prop-kine-cs-wp}, we know $f^{n+1}$ is well-defined. Substituting $f^{n+1}$ into $\eqref{eq-cs-s-appro}_2$, $(\bu^{n+1},\nabla P^{n+1})$ can be obtained by solving the linear equations in routine way. Thus $(f^{n+1},\bu^{n+1},\nabla P^{n+1})$ is obtained, and repeating the above procedure gives rise to an approximate solution sequence. In the iteration procedure, $\bu^0$ is set by
\begin{equation} \label{eq-cs-s-approini}
     \begin{dcases}
       \bu^0_t=\Delta\bu^0,\\
       \bu^0|_{t=0}=\bu_0.
     \end{dcases}
\end{equation}
It is easy to see that
\begin{equation}\label{eq-induini}
 \sup_{0\le t\le \infty}\|\bu^0(t)\|_{H^1}+\|\bu^0_t\|_{L^2(0,\infty;L^2)}+\|\nabla^2\bu^0\|_{L^2(0,\infty;L^2)}+ \|\nabla^2\bu^0\|_{L^s(0,\infty;L^q)}\le 2K_0 \|\bu_0\|_{H^1}
\end{equation}
for some constant $K_0>0$, where we have applied Proposition \ref{prop-s} to \eqref{eq-cs-s-approini}. Then the proof of Proposition \ref{prop-loc-exist} is divided into the following three steps.

\vskip 0.3cm
\noindent \textbf{Step 1. Uniform bound on approximate solutions}
\vskip 0.3cm
We prove by induction that there exists $T_* \in (0,T]$, to be determined later, such that
\begin{equation}\label{eq-indu-assup}
 \sup_{0\le t\le T_*}\|\bu^n(t)\|_{H^1}+\|\bu^n_t\|_{L^2(0,T_*;L^2)}+\|\nabla^2\bu^n\|_{L^2(0,T_*;L^2)}+ \|\nabla^2\bu^n\|_{L^s(0,T_*;L^q)}\le 2K_0 \|\bu_0\|_{H^1}
\end{equation}
for some constant $K_0>0$. Under the above induction hypothesis, it suffices to prove that \eqref{eq-indu-assup} holds for $n+1$. Denote by $R^{n+1}(t)$ the bound of $v$-support of $f^{n+1}(t,x,v)$. We deduce from \eqref{eq-cs-estvbd} and Proposition \ref{prop-s}(ii) that
\begin{equation}\label{eq-appro-f-norm}
  \begin{aligned}
    &\sup_{0\le t \le T_1}R^{n+1}(t)\le \left(R_0+\int_0^{T_1} \|\bu^n(t)\|_{L^{\infty}}dt\right)(1+T_1),\\
    &\sup_{0\le t \le T_1}\|f^{n+1}(t)\|_{H^1_{\omega}}\le \|f_0\|_{H_{\omega}^1}\exp \left( C \int_0^{T_1} \Big(1+R^{n+1}(t)+\|\bu^n(t)\|_{W^{1,\infty}} \Big) dt \right).
  \end{aligned}
\end{equation}
Take $T_1:=T_1(R_0, \|\bu_0\|_{H^1})$ suitably small such that
\begin{equation}\label{eq-appro-f-normtwo}
  \begin{aligned}
    &\sup_{0\le t \le T_1}R^{n+1}(t)\le 2R_0,\\
    &\sup_{0\le t \le T_1}\|f^{n+1}(t)\|_{H^1_{\omega}}\le 2\|f_0\|_{H_{\omega}^1}.
  \end{aligned}
\end{equation}
We take the inner product of $\eqref{eq-cs-s-appro}_2$ with $2\bu^{n+1}$  to obtain
\begin{equation}\label{eq-appro-usqua}
  \begin{aligned}
    \frac{d}{dt}\|\bu^{n+1}\|_{L^2}^2 +2\|\nabla \bu^{n+1}\|_{L^2}^2
    =&2\int_{\bbr^3}\int_{\bbr^3}(v-\bu^{n+1}) \cdot \bu^{n+1}f^{n+1}dv dx\\
    \le&2\int_{\bbr^3}\int_{\bbr^3}v \cdot \bu^{n+1}f^{n+1}dv dx.
  \end{aligned}
\end{equation}
Taking the inner product of $\eqref{eq-cs-s-appro}_2$ with $-\Delta\bu^{n+1}$,  we have
\begin{equation}\label{eq-appro-dif-usqua}
  \begin{aligned}
    &\frac{d}{dt}\|\nabla\bu^{n+1}\|_{L^2}^2 +\|\Delta \bu^{n+1}\|_{L^2}^2\\
    =&\int_{\bbr^3}\bu^n\cdot\nabla\bu^{n+1}\cdot\Delta \bu^{n+1} dx-  \int_{\bbr^3}\int_{\bbr^3}(v-\bu^{n+1}) \cdot \Delta\bu^{n+1}f^{n+1}dv dx\\
    \le&C\|\bu^n\|_{L^{\infty}}^2 \|\nabla\bu^{n+1}\|_{L^2}^2+  C\left\|\int_{\bbr^3} (v-\bu^{n+1})f^{n+1} dv\right\|_{L^2}^2+ \frac14\|\Delta \bu^{n+1}\|_{L^2}^2.
  \end{aligned}
\end{equation}
Taking the dot product of $\eqref{eq-cs-s-appro}_2$ with $2\bu^{n+1}_t$ and integrating the resulting equation over $\bbr^3$, we deduce that
\begin{equation}\label{eq-appro-dif-gradusqua}
  \begin{aligned}
    &\frac{d}{dt}\|\nabla \bu^{n+1}\|_{L^2}^2 +2\|\bu^{n+1}_t\|_{L^2}^2\\
    =&-2\int_{\bbr^3}\bu^n \cdot \nabla \bu^{n+1} \cdot \bu^{n+1}_t dx + 2\int_{\bbr^3}\int_{\bbr^3}(v-\bu^{n+1}) \cdot \bu^{n+1}_t f^{n+1} dv dx\\
    \le& 2\|\bu^{n}\|_{L^{\infty}} \|\nabla \bu^{n+1}\|_{L^2} \|\bu^{n+1}_t\|_{L^2}+ 2\left\|\int_{\bbr^3} (v-\bu^{n+1})f^{n+1} dv\right\|_{L^2} \|\bu^{n+1}_t\|_{L^2}\\
    \le&C\|\bu^{n}\|_{L^{\infty}}^2\|\nabla \bu^{n+1}\|_{L^2}^2+ C\left\|\int_{\bbr^3} (v-\bu^{n+1})f^{n+1} dv\right\|_{L^2}^2+ \|\bu_t^{n+1}\|_{L^2}^2,
  \end{aligned}
\end{equation}
Summing \eqref{eq-appro-usqua}, \eqref{eq-appro-dif-usqua} and \eqref{eq-appro-dif-gradusqua} together yields
\begin{equation}\label{eq-appro-dif-uhonesqua}
  \begin{aligned}
    &\frac{d}{dt}\|\bu^{n+1}\|_{H^1}^2 +2\|\nabla\bu^{n+1}\|_{L^2}^2+\frac34\|\Delta \bu^{n+1}\|_{L^2}^2 +\|\bu_t^{n+1}\|_{L^2}^2\\
    \le&C\|\bu^{n}\|_{L^{\infty}}^2\|\nabla \bu^{n+1}\|_{L^2}^2+ C\left\|\int_{\bbr^3} (v-\bu^{n+1})f^{n+1} dv\right\|_{L^2}^2\\
     &+ 2\int_{\bbr^3}\int_{\bbr^3}v \cdot \bu^{n+1}f^{n+1}dv dx.
  \end{aligned}
\end{equation}
Take $T_2\le T_1$. It holds for $t\in [0, T_2]$ that
\begin{equation}\label{eq-coultwo}
  \begin{aligned}
    &C\left\|\int_{\bbr^3} (v-\bu^{n+1})f^{n+1} dv\right\|_{L^2}^2\\
    \le&C(R_0)\|f^{n+1}\|_{L_{\omega}^2}^2+ C(R_0)\|f^{n+1}\|_{L_{\omega}^2}^2\|\nabla \bu^{n+1}\|_{L^2}\|\nabla^2 \bu^{n+1}\|_{L^2}\\
    \le&C(R_0)\|f^{n+1}\|_{L_{\omega}^2}^2+C(R_0)\|f^{n+1}\|_{L_{\omega}^2}^4\|\nabla \bu^{n+1}\|_{L^2}^2 +\frac14\|\Delta \bu^{n+1}\|_{L^2}^2,
  \end{aligned}
\end{equation}
and
\begin{equation}\label{eq-couint}
  \begin{aligned}
    2\int_{\bbr^3}\int_{\bbr^3}v \cdot \bu^{n+1}f^{n+1}dv dx \le& C(R_0)\|f^{n+1}\|_{L_{\omega}^2}\|\bu^{n+1}\|_{L^2}\\
     \le&C(R_0)\|f^{n+1}\|_{L_{\omega}^2}^2+ \|\bu^{n+1}\|_{L^2}^2,
  \end{aligned}
\end{equation}
where we have used the Sobolev inequality
\[
 \|\bu^{n+1}\|_{L^{\infty}}\le C \|\nabla \bu^{n+1}\|_{L^2}^{\frac12}\|\nabla^2 \bu^{n+1}\|_{L^2}^{\frac12}
\]
and the elliptic estimate
\[
 \|\nabla^2 \bu^{n+1}\|_{L^2}\le C\|\Delta \bu^{n+1}\|_{L^2}.
\]
Substituting \eqref{eq-coultwo} and \eqref{eq-couint} into \eqref{eq-appro-dif-uhonesqua} leads to
\begin{equation}\label{eq-appro-gron-uhonesqua}
  \begin{aligned}
    &\frac{d}{dt}\|\bu^{n+1}\|_{H^1}^2 +\frac12\|\Delta \bu^{n+1}\|_{L^2}^2 +\|\bu_t^{n+1}\|_{L^2}^2\\
    \le&\Big(C+C\|\bu^{n}\|_{L^{\infty}}^2+C(R_0)\|f^{n+1}\|_{L_{\omega}^2}^4\Big) \|\bu^{n+1}\|_{H^1}^2+ C(R_0)\|f^{n+1}\|_{L_{\omega}^2}^2
  \end{aligned}
\end{equation}
for $t\in [0, T_2]$. Solving the above Gronwall's inequality gives
\begin{equation}\label{eq-appro-uhonesqua}
  \begin{aligned}
    &\sup_{0\le t\le T_2}\|\bu^{n+1}(t)\|_{H^1}^2 +\frac12\int_0^{T_2}\|\Delta \bu^{n+1}\|_{L^2}^2 dt + \int_0^{T_2}\|\bu_t^{n+1}\|_{L^2}^2dt\\
    \le&\left(\|\bu_0\|_{H^1}^2+\int_0^{T_2}C(R_0)\|f^{n+1}\|_{L_{\omega}^2}^2dt\right)\exp\left(\int_0^{T_2}\Big(C+C\|\bu^{n}\|_{L^{\infty}}^2
    +C(R_0)\|f^{n+1}\|_{L_{\omega}^2}^4\Big)dt\right)
  \end{aligned}
\end{equation}
Take $T_2:=T_2(R_0, \|\bu_0\|_{H^1}, \|f_0\|_{H_{\omega}^1})$ suitably small such that
\begin{equation}\label{eq-u-nabtu-ut}
 \sup_{0\le t\le T_2}\|\bu^{n+1}(t)\|_{H^1}+\|\bu^{n+1}_t\|_{L^2(0,T_2;L^2)}+\|\nabla^2\bu^{n+1}\|_{L^2(0,T_2;L^2)}\le K_0 \|\bu_0\|_{H^1}.
\end{equation}
Take $T_3\le T_2$. Applying Proposition \ref{prop-s} to $\eqref{eq-cs-s-appro}_2-\eqref{eq-cs-s-appro}_3$ results in
\begin{equation}\label{eq-appro-ulsq}
  \begin{aligned}
    \|\nabla^2\bu^{n+1}\|_{L^s(0,T_3;L^q)}
    \le&C\Bigg(\|\bu_0\|_{H^1}+\|\bu^n\cdot\nabla\bu^{n+1}\|_{L^s(0,T_3;L^q)}\\
     &\quad+\left\|\int_{\bbr^3} (v-\bu^{n+1})f^{n+1} dv\right\|_{L^s(0,T_3;L^q)}\Bigg).
  \end{aligned}
\end{equation}
The last two terms in the right-hand side of \eqref{eq-appro-ulsq} are estimated as follows.
\begin{equation}\label{eq-appro-ulsq-p1}
  \begin{aligned}
    &C\|\bu^n\cdot\nabla\bu^{n+1}\|_{L^s(0,T_3;L^q)}\\
    \le&C\left(\int_0^{T_3}\|\bu^{n}\|_{L^{\infty}}^s\|\nabla\bu^{n+1}\|_{L^2}^{s(1-\theta)}\|\nabla^2\bu^{n+1}\|_{L^q}^{s\theta} dt\right)^{\frac1s}\\
     \le&C\left(\int_0^{T_3}\|\bu^{n}\|_{L^{\infty}}^{\frac{s}{1-\theta}}\|\nabla\bu^{n+1}\|_{L^2}^{s} dt\right)^{\frac1s}+\frac12 \|\nabla^2\bu^{n+1}\|_{L^s(0,T_3;L^q)}\\
     \le&C \sup_{0\le t\le T_3}\|\nabla\bu^{n+1}(t)\|_{L^2} \left(\int_0^{T_3}\|\bu^{n}\|_{L^{\infty}}^{\frac{s}{1-\theta}} dt\right)^{\frac1s}+\frac12 \|\nabla^2\bu^{n+1}\|_{L^s(0,T_3;L^q)},
  \end{aligned}
\end{equation}
where we have used the following Gagliardo--Nirenberg inequality in $\bbr^3$,
\begin{equation}\label{eq-gn-ineq}
  \|\nabla\bu^{n+1}\|_{L^q}\le C\|\nabla\bu^{n+1}\|_{L^2}^{1-\theta}\|\nabla^2\bu^{n+1}\|_{L^q}^{\theta} \quad \text{with $1-\frac3q=-\frac{1-\theta}{2}+\theta\left(2-\frac3q\right)$}.
\end{equation}
Using \eqref{eq-appro-f-normtwo} and the Sobolev inequality, we proceed to estimate the last term in the right-hand side of \eqref{eq-appro-ulsq}.
\begin{equation}\label{eq-appro-ulsq-p2}
  \begin{aligned}
   &C\left\|\int_{\bbr^3} (v-\bu^{n+1})f^{n+1} dv\right\|_{L^s(0,T_3;L^q)}\\
   \le& C(R_0)\sup_{0\le t\le T_3}\|f^{n+1}(t)\|_{H_{\omega}^1}T_3^{\frac1s}+ C(R_0)\sup_{0\le t\le T_3}\|f^{n+1}(t)\|_{H_{\omega}^1}
  \left(\int_0^{T_3}\|\bu^{n+1}\|_{L^{\infty}}^{s} dt\right)^{\frac1s}.
  \end{aligned}
\end{equation}
Substituting \eqref{eq-appro-ulsq-p1} and \eqref{eq-appro-ulsq-p2} into \eqref{eq-appro-ulsq} gives
\begin{equation}\label{eq-appro-ulsq-subst}
  \begin{aligned}
    &\|\nabla^2\bu^{n+1}\|_{L^s(0,T_3;L^q)}\\
    \le&C\|\bu_0\|_{H^1}+C \sup_{0\le t\le T_3}\|\nabla\bu^{n+1}(t)\|_{L^2} \left(\int_0^{T_3}\|\bu^{n}\|_{L^{\infty}}^{\frac{s}{1-\theta}} dt\right)^{\frac1s}
    \\&+ C(R_0)\sup_{0\le t\le T_3}\|f^{n+1}(t)\|_{H_{\omega}^1}T_3^{\frac1s}+ C(R_0)\sup_{0\le t\le T_3}\|f^{n+1}(t)\|_{H_{\omega}^1}
  \left(\int_0^{T_3}\|\bu^{n+1}\|_{L^{\infty}}^{s} dt\right)^{\frac1s}.
  \end{aligned}
\end{equation}
It is easy to see from \eqref{eq-sq-rela} and \eqref{eq-gn-ineq} that
\[
 \frac{s}{1-\theta}=\frac{(5q-6)s}{2q}<2 \quad \text{and} \quad 1<s<2.
\]
Using \eqref{eq-indu-assup}, \eqref{eq-appro-f-normtwo} and \eqref{eq-u-nabtu-ut}, we take $T_3:=T_3(R_0, \|\bu_0\|_{H^1}, \|f_0\|_{H_{\omega}^1})$ suitably small such that
\begin{equation}\label{eq-appro-ulsqk}
  \|\nabla^2\bu^{n+1}\|_{L^s(0,T_3;L^q)}\le K_0\|\bu_0\|_{H^1}.
\end{equation}
Define $T_*:=\min\{T_1, T_2, T_3\}$. Adding \eqref{eq-appro-ulsqk} to \eqref{eq-u-nabtu-ut} results in
\begin{equation}\label{eq-indu-assupnex}
 \sup_{0\le t\le T_*}\|\bu^{n+1}(t)\|_{H^1}+\|\bu^{n+1}_t\|_{L^2(0,T_*;L^2)}+\|\nabla^2\bu^{n+1}\|_{L^2(0,T_*;L^2)}+ \|\nabla^2\bu^{n+1}\|_{L^s(0,T_*;L^q)}\le 2K_0 \|\bu_0\|_{H^1}.
\end{equation}
From \eqref{eq-induini}, we know \eqref{eq-indu-assup} holds for $n=0$. Thus, we prove by induction that \eqref{eq-indu-assup} holds for all $n \in \bbn$.
\vskip 3mm
\noindent \textbf{Step 2. Convergence of approximate aolutions}
\vskip 3mm
Define
\[
 \overline{f}^{n+1}:=f^{n+1}-f^n,\quad \overline{\bu}^{n+1}:=\bu^{n+1}-\bu^n,\quad \overline{P}^{n+1}:=P^{n+1}-P^n.
\]
It follows from \eqref{eq-cs-s-appro}-\eqref{eq-cs-s-approcau} that
\begin{equation} \label{eq-cs-s-appro-dif}
     \begin{dcases}
      \overline{f}^{n+1}_t +\bu^n \cdot\nabla_{x} \overline{f}^{n+1} +\nabla_{v}\cdot\Big[L[f^{n+1}]\overline{f}^{n+1}+(\bu^n-v)\overline{f}^{n+1}\Big]\\
           \qquad+\overline{\bu}^{n}\cdot\nabla_{x}f^n+\nabla_{v}\cdot\Big[L[\overline{f}^{n+1}]f^n +f^n \overline{\bu}^n\Big]=0,\\
       \overline{\bu}^{n+1}_t +\bu^n\cdot \nabla \overline{\bu}^{n+1} +\overline{\bu}^n \cdot \nabla \bu^n+\nabla\overline{P}^{n+1}\\
          \qquad=\Delta \overline{\bu}^{n+1}  -\int_{\bbr^2}f^{n+1}\overline{\bu}^{n+1}dv +\int_{\bbr^2}(v-\bu^{n})\overline{f}^{n+1} dv,\\
       \nabla\cdot\overline{\bu}^{n+1}=0,
     \end{dcases}
\end{equation}
with
\begin{equation} \label{eq-cs-s-approini-dif}
 \overline{f}^{n+1}|_{t=0}=0,  \quad \overline{\bu}^{n+1}|_{t=0}=0.
\end{equation}
Taking the inner product of $\eqref{eq-cs-s-appro-dif}_2$ with $\overline{\bu}^{n+1}$, we deduce that for $t\in [0, T_*]$,
\begin{equation}\label{eq-app-udif-squa}
  \begin{aligned}
    &\frac12\frac{d}{dt}\|\overline{\bu}^{n+1}\|_{L^2}^2 +\|\nabla\overline{\bu}^{n+1}\|_{L^2}^2\\
    \le& -\int_{\bbr^3} \overline{\bu}^n \cdot \nabla \bu^n\cdot \overline{\bu}^{n+1}dx+ \int_{\bbr^3}\int_{\bbr^3}\overline{f}^{n+1}(v-\bu^{n})dv\cdot\overline{\bu}^{n+1}dx\\
   \le&C\|\nabla\overline\bu^{n}\|_{L^2}\|\nabla\bu^{n}\|_{L^3}
   \|\overline{\bu}^{n+1}\|_{L^2}+C\left\|\int_{\bbr^3}\overline{f}^{n+1}v dv\right\|_{L^{\frac65}} \|\nabla\overline{\bu}^{n+1}\|_{L^2}\\ &+C\left\|\int_{\bbr^3}\overline{f}^{n+1} dv\right\|_{L^{\frac65}}\|\bu^{n}\|_{L^{\infty}}\|\nabla\overline{\bu}^{n+1}\|_{L^2}\\
   \le&C\|\nabla\bu^{n}\|_{L^3}^2\|\overline{\bu}^{n+1}\|_{L^2}^2 +C(R_0)\Big(1+\|\bu^{n}\|_{L^{\infty}}^2\Big)\left\|\overline{f}^{n+1}\right\|_{L^{\frac65}}^2\\   &+\frac12 \|\nabla\overline{\bu}^{n+1}\|_{L^2}^2 +\frac{1}{24} \|\nabla\overline{\bu}^{n}\|_{L^2}^2,
  \end{aligned}
\end{equation}
that is,
\begin{equation}\label{eq-app-udif-squashort}
 \begin{gathered}
  \frac{d}{dt}\|\overline{\bu}^{n+1}\|_{L^2}^2 +\|\nabla\overline{\bu}^{n+1}\|_{L^2}^2
   \le C\|\nabla\bu^{n}\|_{L^3}^2\|\overline{\bu}^{n+1}\|_{L^2}^2 \\ +C(R_0)\Big(1+\|\bu^{n}\|_{L^{\infty}}^2\Big)\left\|\overline{f}^{n+1}\right\|_{L^{\frac65}}^2 +\frac{1}{12} \|\nabla\overline{\bu}^{n}\|_{L^2}^2.
   \end{gathered}
\end{equation}
Multiplying $\eqref{eq-cs-s-appro-dif}_1$ by $\frac65 \left|\overline{f}^{n+1}\right|^{\frac15}\text{sgn}\overline{f}^{n+1}$ leads to
\begin{equation}\label{eq-app-kcswt-dif-frac}
  \begin{aligned}
    &\frac{\partial}{\partial t}\left|\overline{f}^{n+1}\right|^{\frac65} +\bu^n \cdot\nabla_{x} \left|\overline{f}^{n+1}\right|^{\frac65} +\nabla_{v}\cdot\bigg[L[f^{n+1}]\left|\overline{f}^{n+1}\right|^{\frac65} +(\bu^n-v)\left|\overline{f}^{n+1}\right|^{\frac65}\bigg]\\
    =&-\frac15 \nabla_v\cdot L[f^{n+1}]\left|\overline{f}^{n+1}\right|^{\frac65} +\frac35\left|\overline{f}^{n+1}\right|^{\frac65}\\
    &-\frac65 \left|\overline{f}^{n+1}\right|^{\frac15}\text{sgn}\overline{f}^{n+1}\overline{\bu}^{n}\cdot\nabla_x f^n\\
    &-\frac65 \left|\overline{f}^{n+1}\right|^{\frac15}\text{sgn}\overline{f}^{n+1}\nabla_{v} \cdot \bigg(L[\overline{f}^{n+1}]f^n +\overline{\bu}^n f^n\bigg).
  \end{aligned}
\end{equation}
Integrating \eqref{eq-app-kcswt-dif-frac} over $\bbr^3\times \bbr^3$ gives
\begin{equation}\label{eq-app-kcswt-dif-frac-gron}
  \begin{aligned}
    &\frac{d}{d t}\left\|\overline{f}^{n+1}\right\|_{L^{\frac65}}^{\frac65}\\
    =&\int_{\bbr^3}\int_{\bbr^3}\Bigg(-\frac15 \nabla_v\cdot L[f^{n+1}]\left|\overline{f}^{n+1}\right|^{\frac65} +\frac35\left|\overline{f}^{n+1}\right|^{\frac65}\Bigg)dx dv \\
    &-\int_{\bbr^3}\int_{\bbr^3}\frac65 \left|\overline{f}^{n+1}\right|^{\frac15}\text{sgn}\overline{f}^{n+1}\overline{\bu}^{n}\cdot\nabla_x f^ndx dv\\
    &-\int_{\bbr^3}\int_{\bbr^3} \frac65 \left|\overline{f}^{n+1}\right|^{\frac15}\text{sgn}\overline{f}^{n+1}\nabla_{v} \cdot \bigg(L[\overline{f}^{n+1}]f^n +\overline{\bu}^n f^n\bigg) dx dv\\
    =:&\sum_{i=1}^3N_i.
  \end{aligned}
\end{equation}
For $t\in [0, T_*]$, we estimate each $N_i$ $(i=1,2,3)$ as follows.
\begin{equation*}
  \begin{aligned}
     N_1:=&\int_{\bbr^3}\int_{\bbr^3}\Bigg(-\frac15 \nabla_v\cdot L[f^{n+1}]\left|\overline{f}^{n+1}\right|^{\frac65} +\frac35\left|\overline{f}^{n+1}\right|^{\frac65}\Bigg)dx dv\\
     \le& C \left\|\overline{f}^{n+1}\right\|_{L^{\frac65}}^{\frac65};\\
     N_2:=& -\int_{\bbr^3}\int_{\bbr^3}\frac65 \left|\overline{f}^{n+1}\right|^{\frac15}\text{sgn}\overline{f}^{n+1}\overline{\bu}^{n}\cdot\nabla_x f^ndx dv\\
         \le& C\left\|\overline{f}^{n+1}\right\|_{L^{\frac65}}^{\frac15} \left\|\nabla_x f^{n}\right\|_{L_x^{\frac32}L_v^{\frac65}} \|\nabla\overline{\bu}^n\|_{L^2}\\
         \le& C(R_0)\|\nabla_x f^{n}\|_{L_{\omega}^2} \left\|\overline{f}^{n+1}\right\|_{L^{\frac65}}^{\frac15} \|\nabla\overline{\bu}^n\|_{L^2};
   \end{aligned}
\end{equation*}
\begin{equation*}
   \begin{aligned}
     N_3:=& -\int_{\bbr^3}\int_{\bbr^3} \frac65 \left|\overline{f}^{n+1}\right|^{\frac15}\text{sgn}\overline{f}^{n+1}\nabla_{v} \cdot \bigg(L[\overline{f}^{n+1}]f^n +\overline{\bu}^n f^n\bigg) dx dv\\
     \le& C\left\|\overline{f}^{n+1}\right\|_{L^{\frac65}}^{\frac15} \left\|\overline{f}^{n+1}\right\|_{L^{1}} \left\|{f}^{n}\right\|_{L^{\frac65}}\\
     &+C(R_0)\left\|\overline{f}^{n+1}\right\|_{L^{\frac65}}^{\frac15} \left\|\overline{f}^{n+1}\right\|_{L^{1}} \left\|\nabla_v{f}^{n}\right\|_{L^{\frac65}}\\
     &+ C\left\|\overline{f}^{n+1}\right\|_{L^{\frac65}}^{\frac15} \left\|\nabla_v f^{n}\right\|_{L_x^{\frac32}L_v^{\frac65}} \|\nabla\overline{\bu}^n\|_{L^2}\\
     \le& C(R_0)\|f^{n}\|_{H_{\omega}^1} \left\|\overline{f}^{n+1}\right\|_{L^{\frac65}}^{\frac15} \left\|\overline{f}^{n+1}\right\|_{L^{1}}+ C(R_0)\|\nabla_v f^{n}\|_{L_{\omega}^2} \left\|\overline{f}^{n+1}\right\|_{L^{\frac65}}^{\frac15} \|\nabla\overline{\bu}^n\|_{L^2}.
  \end{aligned}
\end{equation*}
In the estimate of $N_2$ and $N_3$, we have used the following inequalities.
\begin{equation*}
  \begin{aligned}
     \left\|f^n\right\|_{L^{\frac65}}\le& C(R_0)\Bigg(\int_{\bbr^3}\int_{\bbr^3}|f^n|^2(1+|x|^2)^{\frac{2\alpha}{3}} dx dv \Bigg)^{\frac12}\|(1+|x|^2)^{-\frac{\alpha}{3}}\|_{L^3}\\
     \le& C(R_0) \|f^n\|_{L_{\omega}^2};
  \end{aligned}
\end{equation*}
\begin{equation*}
  \begin{aligned}
     \left\|\nabla_x f^{n}\right\|_{L_x^{\frac32}L_v^{\frac65}}
     \le& C(R_0)\Bigg(\int_{\bbr^3}\int_{\bbr^3}|\nabla_{x}f^n|^2(1+|x|^2)^{\frac{\alpha}{3}} dx dv \Bigg)^{\frac12} \|(1+|x|^2)^{-\frac{\alpha}{6}}\|_{L^6}\\
     \le& C(R_0) \|\nabla_{x}f^n\|_{L_{\omega}^2}.
   \end{aligned}
\end{equation*}
$\displaystyle\left\|\nabla_v{f}^{n}\right\|_{L^{\frac65}}$ and $\displaystyle\left\|\nabla_v f^{n}\right\|_{L_x^{\frac32}L_v^{\frac65}}$ are estimated similarly as above.
Substituting the estimates on $N_i$ $(i=1,2,3)$ into \eqref{eq-app-kcswt-dif-frac-gron}, we deduce that for $t\in [0, T_*]$,
\begin{equation}\label{eq-app-kcswt-dif-frac-gronshort}
  \begin{aligned}
    \frac{d}{d t}\left\|\overline{f}^{n+1}\right\|_{L^{\frac65}}^{2}
    \le&C(R_0)\bigg(1 + \|f^n\|_{H_{\omega}^1}^2 \bigg)\left\|\overline{f}^{n+1}\right\|_{L^{\frac65}}^{2}
    +\left\|\overline{f}^{n+1}\right\|_{L^1}^2 +\frac{1}{12} \|\nabla\overline{\bu}^n\|_{L^2}^2.
  \end{aligned}
\end{equation}
Similarly, we have for $t\in [0, T_*]$,
\begin{equation}\label{eq-app-kcswt-dif-lonesqu-gron}
  \begin{aligned}
    \frac{d}{d t}\left\|\overline{f}^{n+1}\right\|_{L^{1}}^{2}
    \le&C(R_0)\bigg(1 + \|f^n\|_{H_{\omega}^1}^2 \bigg)\left\|\overline{f}^{n+1}\right\|_{L^{1}}^{2}
     +\frac{1}{12} \|\nabla\overline{\bu}^n\|_{L^2}^2.
  \end{aligned}
\end{equation}
Define
\[
 F^{n+1}(t):=\|\overline{\bu}^{n+1}\|_{L^2}^2 +\left\|\overline{f}^{n+1}\right\|_{L^{\frac65}}^2 +\left\|\overline{f}^{n+1}\right\|_{L^1}^2.
\]
Combining \eqref{eq-app-udif-squashort}, \eqref{eq-app-kcswt-dif-frac-gronshort}, and \eqref{eq-app-kcswt-dif-lonesqu-gron} , we deduce that for $t\in [0, T_*]$,
\begin{equation}\label{eq-app-dif-gron}
  \begin{aligned}
    &\frac{d}{dt}F^{n+1} +\|\nabla\overline{\bu}^{n+1}\|_{L^2}^2\\
    \le& C(R_0)\bigg(1 +\|\bu^n\|_{L^{\infty}}^2 +\|f^n\|_{H_{\omega}^1}^2 +\|\nabla{\bu}^{n}\|_{L^3}^2 \bigg) F^{n+1}
     +\frac14 \|\nabla\overline{\bu}^n\|_{L^2}^2.
  \end{aligned}
\end{equation}
Solving the above Gronwall inequality in $[0, T_0]$ for $0<T_0\le T_*$, we obtain
\begin{equation}\label{eq-app-dif-est}
 \begin{aligned}
  \sup_{0\le t\le T_0}F^{n+1}(t) +\int_0^{T_0}\|\nabla\overline{\bu}^{n+1}(t)\|_{L^2}^2 dt
  \le \frac{A(T_0)}{4} \int_0^{T_0}\|\nabla\overline{\bu}^n(t)\|_{L^2}^2 dt,
  \end{aligned}
\end{equation}
where $A(T_0)$ is given by
\[
 A(T_0):=\exp\left(\int_0^{T_0}  C(R_0)\bigg(1 +\|\bu^n\|_{L^{\infty}}^2 +\|f^n\|_{H_{\omega}^1}^2 +\|\nabla{\bu}^{n}\|_{L^3}^2 \bigg)  dt \right).
\]
Using the uniform bound on approximate solutions, we take $T_0:=T_0(R_0, \|\bu_0\|_{H^1}, \|f_0\|_{H_{\omega}^1})$ suitably small, so that
\[
 A(T_0) \le 2.
\]
Thus, we have
\begin{equation}\label{eq-app-dif-est-tzro}
  \sup_{0\le t\le T_0}F^{n+1}(t) +\int_0^{T_0}\|\nabla\overline{\bu}^{n+1}(t)\|_{L^2}^2 dt
  \le \frac{1}{2} \int_0^{T_0}\|\nabla\overline{\bu}^n(t)\|_{L^2}^2 dt.
\end{equation}
Summing \eqref{eq-app-dif-est-tzro} over all $n\in \bbn$ yields
\begin{equation}\label{eq-app-dif-est-sum}
  \sum_{n=2}^{\infty}\sup_{0\le t\le T_0}F^n(t) +\frac{1}{2}\sum_{n=2}^{\infty} \int_0^{T_0}\|\nabla\overline{\bu}^n(t)\|_{L^2}^2 dt
  \le \frac{1}{2}\int_0^{T_0}\|\nabla\overline{\bu}^1(t)\|_{L^2}^2 dt.
\end{equation}
We deduce from \eqref{eq-app-dif-est-sum} that there exists $(f, \bu)$ such that
\begin{equation}\label{eq-app-conver-lowoder}
  \begin{aligned}
    &f^n \to f, \quad\text{in $C([0, T_0]; L^1)$, as $n \to \infty$};\\
    &\bu^n \to \bu, \quad \text{in $C(0, T_0; L^2)$, as $n \to \infty$};\\
    &\bu^n \to \bu, \quad \text{in $L^2(0, T_0; D^1)$, as $n \to \infty$}.
  \end{aligned}
\end{equation}
From \eqref{eq-app-conver-lowoder}, it is easy to show that $(f, \bu)$ verifies \eqref{eq-cs-ns} in the sense of distributions.
\vskip 3mm
\noindent \textbf{Step 3. Continuity in strong solution space}
\vskip 3mm
Combining the uniform bound estimates \eqref{eq-indu-assup} and \eqref{eq-appro-f-normtwo} with \eqref{eq-app-conver-lowoder}, we deduce that
\begin{equation}\label{eq-app-conver-wek}
  \begin{aligned}
    f^n &\rightharpoonup f, \quad\text{weakly-$\star$ in $L^{\infty}(0, T_0; H^1_{\omega})$, as $n \to \infty$};\\
    \bu^n &\rightharpoonup  \bu, \quad \text{weakly in $L^{2}(0, T_0; H^2)$, as $n \to \infty$};\\
    \bu^n_t &\rightharpoonup \bu_t, \quad \text{weakly in $L^2(0, T_0; L^2)$, as $n \to \infty$};\\
    \nabla^2\bu^n &\rightharpoonup \nabla^2\bu, \quad \text{weakly in $L^s(0, T_0; L^{q})$, as $n \to \infty$}.
  \end{aligned}
\end{equation}
It follows from \eqref{eq-app-conver-wek} that $\bu_t \in L^2(0, T_0; L^2)$ and $\bu \in L^2(0, T_0; H^{2})$.
Using the regularity of $\bu$, we can show that $\bu\in C(0, T_0; H^1)$. Thus,
\[
 \bu \in C([0, T_0]; H^1)\cap L^2(0, T_0; D^{2})\cap L^s(0, T_0; D^{2,q}).
\]
Employing Proposition \ref{prop-kine-cs-wp}, we infer that
\begin{equation}\label{eq-conti-fwtsobo}
 f \in C([0, T_0]; H^1_{\omega}).
\end{equation}
Therefore, $(f,\bu)$ is the desired strong solution in the sense of Definition \ref{def-stro}. The uniqueness of strong solutions can be proved in the same way as in the derivation of \eqref{eq-app-dif-gron}.
This completes the proof. $\hfill \square$
%%%%%%%%%%%%%%%%%%%%%%%%%%%%%%%%%%%%%%%%%%%%%%%%%%%%%%%%%%%%%%%%%%%%%%%%%%%%%%%%%%%%%%%%%%%%%%%%%%%%%%%%%%%%%%
%
%                      Sect.4 A Priori Estimates
%
%%%%%%%%%%%%%%%%%%%%%%%%%%%%%%%%%%%%%%%%%%%%%%%%%%%%%%%%%%%%%%%%%%%%%%%%%%%%%%%%%%%%%%%%%%%%%%%%%%%%%%%%%%%%%%
\section{A Priori Estimates}\label{sec-apriori}
\setcounter{equation}{0}
In this section, we derive some a priori estimates on classical solutions to the coupled model, which are used to extend the local strong solutions.
\begin{lemma}\label{lm-cs-s-emapriori}
Let $0<R_0<\infty$. Assume that $f_0(x,v)\ge 0$, $f_0(x,v) \in H_{\omega}^1(\bbr^3 \times \bbr^3)\cap L^{\infty}(\bbr^3 \times \bbr^3)$, and $\bu_0(x) \in H^{1}(\bbr^3)$,
with $\nabla \cdot \bu_0=0$ and the $v$-support of $f_0(x,v)$ satisfying
\[
 \text{supp}_{v}f_0(x,\cdot)\subseteq B(R_0) \quad \text{for a.e. $x \in \bbr^3$}.
\]
If $(f, \bu)$ is a classical solution in $[0, T]$, then it holds that
\[
 \begin{aligned}
   (1)&\ f\ge0 \ \text{and} \ \sup_{0\le t\le T}\|f(t)\|_{L^{\infty}}\le \|f_0\|_{L^{\infty}}\exp\big(CT\big),\quad \text{where $C:=C\Big(\|\varphi\|_{C^0}, \|f_0\|_{L^1}\Big)$};\\
   (2)&\ \sup_{0\le t\le T}\|\rho(t)\|_{L^{p}}=\|\rho_0\|_{L^p}\le C\|f_0\|_{L^{\infty}}^{1-\frac1p}R_0^{3-\frac3p},\quad  1\le p\le\infty, \\
    &\ \text{for some constant $C>0$};\\
   (3)&\ E(T)+\frac12\int_0^T \int_{\bbr^6}\int_{\bbr^6}\varphi(|x-y|)f(t,y,v^*)f(t,x,v)|v^*-v|^2dy dv^* dx dv dt\\
   & \qquad + \int_0^T \|\nabla \bu(t) \|_{L^2}^2 d t +\int_0^T \int_{\bbr^6} f(t,x,v)|\bu-v|^2 dx dv dt=E_0.
 \end{aligned}
\]
\end{lemma}
\begin{proof}
(1)\ Denote by $(X(t;x_0,v_0),V(t;x_0,v_0))$ the characteristic to $\eqref{eq-cs-ns}_1$, emanating from $(x_0,v_0)$. It verifies that
\begin{equation} \label{eq-charac}
\begin{dcases}
  \frac{d X}{d t}=\bu(t,X), \\
  \frac{d V}{d t}=\int_{\bbr^{6}} \varphi(|X-y|)f(t, y,v^*)(v^*-V)dy dv^*+\bu(t,X)-V.
\end{dcases}
\end{equation}
\begin{equation} \label{eq-charac-ini}
 X(0;x_0,v_0)=x_0, \qquad V(0;x_0,v_0)=v_0.
\end{equation}
Recall that
\[
 a(t,x)=\int_{\bbr^{6}} \varphi(|x-y|)f(t, y,v^*) dy dv^*,
\]
\[
 \mathbf{b} (t,x)=\int_{\bbr^{6}} \varphi(|x-y|)f(t, y, v^*)v^* dy dv^*.
\]
Solving the equation \eqref{eq-kine-cs} by the method of characteristics gives
\begin{equation} \label{eq-kin-cs-positive}
 f(t,X(t;x_0,v_0),V(t;x_0,v_0))=f_0(x_0,v_0)\exp \left(3 \int_0^t \Big[1+a(\tau,X(\tau))\Big] d \tau \right)\ge 0.
\end{equation}
From \eqref{eq-kin-conser-mass}, \eqref{eq-kin-cs-positive} and the initial condition $f_0(x, v)\in L^{\infty}(\bbr^3\times\bbr^3)$, we deduce that
\begin{equation} \label{eq-kin-cs-linfty}
  \sup_{0\le t\le T}\|f(t)\|_{L^{\infty}}\le \|f_0\|_{L^{\infty}}\exp\big(CT\big),\quad \text{where $C:=C\Big(\|\varphi\|_{C^0}, \|f_0\|_{L^1}\Big)$}.
\end{equation}
(2)\ Integrating $\eqref{eq-cs-ns}_1$ with respect to $v$ over $\bbr^3$ gives
\begin{equation} \label{eq-parden-mascon}
  \rho_t+\bu\cdot\nabla\rho=0.
\end{equation}
Multiplying \eqref{eq-parden-mascon} by $p\rho^{p-1}$, $1\le p<\infty$, and integrating the resulting equation over $[0,t]\times\bbr^3$, we obtain that for $t\in [0,T]$
\begin{equation} \label{eq-parden-pnorm}
  \|\rho(t)\|_{L^p}=\|\rho_0\|_{L^p}\le \|\rho_0\|_{L^1}^{\frac1p}\|\rho_0\|_{L^\infty}^{1-\frac1p}\le C\|f_0\|_{L^{\infty}}^{1-\frac1p}R_0^{3-\frac3p}
\end{equation}
for some constat $C>0$, where we have used the incompressible condition $\nabla\cdot\bu=0$.
Since
\begin{equation} \label{eq-parden-inftynorm}
  \|\rho(t)\|_{L^\infty}=\lim_{p\to\infty}\|\rho(t)\|_{L^p}=\lim_{p\to\infty}\|\rho_0\|_{L^p}=\|\rho_0\|_{L^\infty} \le C\|f_0\|_{L^{\infty}}R_0^{3},
\end{equation}
we combine \eqref{eq-parden-pnorm} with \eqref{eq-parden-inftynorm} to obtain that for all $1\le p\le\infty$,
\[
 \sup_{0\le t\le T}\|\rho(t)\|_{L^{p}}=\|\rho_0\|_{L^p}\le C\|f_0\|_{L^{\infty}}^{1-\frac1p}R_0^{3-\frac3p},\quad \text{for some constant $C>0$.}
\]
(3)\  Multiplying $\eqref{eq-cs-ns}_1$ by $\frac12 |v|^2$, and integrating the resulting equation over $\bbr^3\times\bbr^3$, we have
\begin{equation}\label{eq-cs-ener-dt}
  \begin{gathered}
    \frac{d}{dt}\int_{\bbr^6}\frac12 f |v|^2dx dv +\frac12 \int_{\bbr^6}\int_{\bbr^6}\varphi(|x-y|)f(t,y,v^*)f(t,x,v)|v^*-v|^2dy dv^* dx dv \\
    =\int_{\bbr^6}fv\cdot(\bu-v)dx dv.
  \end{gathered}
\end{equation}
Taking the dot product of $\eqref{eq-cs-ns}_2$ with $\bu$, and integrating the resulting equation over $\bbr^3$, we deduce that
\begin{equation}\label{eq-s-ener-dt}
    \frac12 \frac{d}{dt}\int_{\bbr^3} |\bu|^2 dx +\|\nabla \bu\|_{L^2}^2
    =\int_{\bbr^6}f\bu\cdot(v-\bu)dx dv.
\end{equation}
Adding \eqref{eq-cs-ener-dt} to \eqref{eq-s-ener-dt}, and integrating the resulting equation over $[0,T]$, we arrive at Lemma \ref{lm-cs-s-emapriori}(3). This completes the proof.
\end{proof}
The following lemma is devoted to estimating $\|\bu\|_{L^3(0,T;L^9)}$. It plays a crucial role in the estimate of $\sup_{0\le t\le T}\|\nabla\bu(t)\|_{L^2}$, which is indispensable in applying Proposition \ref{prop-s} to the incompressible Navier--Stokes equations.
\begin{lemma}\label{lm-u-serr-cond}
Under the conditions in Theorem \ref{thm-exist}, if $(f, \bu)$ is a classical solution to \eqref{eq-cs-ns}-\eqref{eq-sys-inidata} in $[0, T]$, then it holds  that
\[
 \|\bu\|_{L^3(0,T;L^9)}^3 \le C\Big(\|\bu_0\|_{L^3}^3+ \|\bu_0\|_{L^3}^6+(1+\|\rho_0\|_{L^{\infty}})E_0 \Big),
\]
for some constant $C>0$ independent of the initial data.
\end{lemma}
\begin{proof}
Using the splitting approach, we decompose $\bu=\bar\bu +\bw$, with $\bar\bu$ and $\bw$ determined by the following equations:
\begin{equation} \label{eq-ns-baru}
     \begin{dcases}
         \bar\bu_t+ \bu \cdot \nabla \bar\bu =\Delta \bar\bu, \\
         \bar\bu|_{t=0}=\bu_0,
     \end{dcases}
\end{equation}
and
\begin{equation} \label{eq-ns-w}
     \begin{dcases}
         \bw_t+ \bu \cdot \nabla \bw +\nabla P=\Delta \bw +\int_{\bbr^3}(v-\bu)fdv, \\
         \bw|_{t=0}=0,
     \end{dcases}
\end{equation}
respectively. Taking the dot product of \eqref{eq-ns-baru} with $3|\bar\bu|\bar\bu$, and integrating the resulting equation over $\bbr^3$, we obtain
\begin{equation}\label{eq-baru-ener-dt}
    \frac{d}{dt} \|\bar\bu\|_{L^3}^3 dx +3\int_{\bbr^3} |\bar\bu||\nabla \bar\bu|^2 dx\le 0,
\end{equation}
where we have used the fact that
\[
 \begin{aligned}
  -\int_{\bbr^3} |\bar\bu|\bar\bu \cdot\Delta\bar\bu dx=&\int_{\bbr^3} \nabla\Big(|\bar\bu|\bar\bu\Big):\nabla\bar\bu dx \ge
  \int_{\bbr^3} |\bar\bu||\nabla \bar\bu|^2 dx,
  \end{aligned}
\]
and
\[
 \begin{aligned}
  3\int_{\bbr^3} \bu \cdot \nabla \bar\bu \cdot \bar\bu |\bar\bu|dx=&\int_{\bbr^3} \bu \cdot \nabla |\bar\bu|^3 dx =0.
  \end{aligned}
\]
Integrating \eqref{eq-baru-ener-dt} over $[0, T]$ gives
\begin{equation}\label{eq-baru-ener}
    \sup_{0\le t\le T} \|\bar\bu(t)\|_{L^3}^3 dx +3\int_0^T\int_{\bbr^3} |\bar\bu||\nabla \bar\bu|^2 dx dt \le \|\bar\bu_0\|_{L^3}^3.
\end{equation}
For convenience of notation, we define $I(\bar\bu)$ as
\[
  I(\bar\bu):=\int_{\bbr^3} |\bar\bu||\nabla \bar\bu|^2 dx,
\]
similarly for the following $I(\bw)$, and $\bh(t,x)$ is defined as
\[
 \bh(t,x):= \int_{\bbr^3} (v-\bu)f dv.
\]
Taking the dot product of \eqref{eq-ns-w} with $3|\bw|\bw$, and integrating the resulting equation over $\bbr^3$, we deduce that for $t\in [0, T]$,
\begin{equation}\label{eq-w-ener-dt}
  \begin{aligned}
    \frac{d}{dt} \|\bw\|_{L^3}^3  +3I(\bw) \le& -3\int_{\bbr^3} |\bw| \bw \cdot\nabla P dx +3\int_{\bbr^3} |\bw| \bw \cdot \bh dx,
  \end{aligned}
\end{equation}
where the pressure $P$ is determined by the following equation:
\[
 \Delta P=-\nabla\cdot(\bu\cdot\nabla\bu)+\nabla\cdot\bh.
\]
Nevertheless, appearance of the additional term $\nabla\cdot\bh$ disables us to use integration by parts to estimate the term $\int_{\bbr^3} |\bw| \bw \cdot\nabla P dx$. We overcome this difficulty by decomposing $P=P_1+P_2$, with $P_1$ and $P_2$ governed by
\[
 \Delta P_1=-\nabla\cdot(\bu\cdot\nabla\bu)\quad\text{and}\quad \Delta P_2=\nabla\cdot\bh,
\]
respectively. Then it follows from the Calder\'on--Zygmund Theorem that
\[
 \|P_1\|_{L^{\frac94}}\le C\|\bu\|_{L^{\frac92}}^{2}\quad \text{and}\quad  \|\nabla P_2\|_{L^{\frac32}}\le C\|\bh\|_{L^{\frac32}},
\]
for some constant $C>0$. We estimate the right-hand side of \eqref{eq-w-ener-dt} as follows.
\begin{equation}\label{eq-w-ener-dt-rh}
  \begin{aligned}
    & -3\int_{\bbr^3} |\bw| \bw \cdot\nabla P dx +3\int_{\bbr^3} |\bw| \bw \cdot \bh dx\\
    \le&-3\int_{\bbr^3} |\bw| \bw \cdot(\nabla P_1+\nabla P_2) dx +3\|\bw\|_{L^6}^2 \|\bh\|_{L^{\frac32}}\\
    \le&3 \int_{\bbr^3} \nabla\cdot\Big(|\bw| \bw\Big) P_1 dx -3\int_{\bbr^3} |\bw| \bw \cdot \nabla P_2 dx +3\|\bw\|_{L^6}^2 \|\bh\|_{L^{\frac32}}\\
    \le&C \int_{\bbr^3} |\bw| |\nabla \bw| |P_1| dx +C\|\bw\|_{L^6}^2 \|\nabla P_2\|_{L^{\frac32}} +3\|\bw\|_{L^6}^2 \|\bh\|_{L^{\frac32}}\\
    \le&C \left\| |\bw|^{\frac12} |\nabla \bw| \right\|_{L^2} \|\bw\|_{L^9}^{\frac12} \|P_1\|_{L^{\frac94}} +C\|\bw\|_{L^6}^2 \|\bh\|_{L^{\frac32}}\\
    \le&CI^{\frac23}(\bw) \|\bu\|_{L^{\frac92}}^{2} +C\|\bw\|_{L^6}^2 \|\bh\|_{L^{\frac32}}\\
    \le&CI^{\frac23}(\bw) \bigg(\|\bar\bu\|_{L^{\frac92}}^{2} +\|\bw\|_{L^{\frac92}}^{2}\bigg) +C\|\bw\|_{L^6}^2 \|\bh\|_{L^{\frac32}}\\
    \le&CI^{\frac23}(\bw) \|\bar\bu\|_{L^3}\|\bar\bu\|_{L^9} +CI^{\frac23}(\bw) \|\bw\|_{L^3}\|\bw\|_{L^9} +C\|\bw\|_{L^3}\|\bw\|_{L^9}^3 +\frac12\|\bh\|_{L^{\frac32}}^2\\
    \le& CI^{\frac23}(\bw) \|\bar\bu\|_{L^3} I^{\frac13}(\bar\bu) +C\|\bw\|_{L^3} I(\bw) +\frac12\|\bh\|_{L^{\frac32}}^2\\
    \le& I(\bw)  +K_1\|\bw\|_{L^3}^3 I(\bw)+C\|\bar\bu\|_{L^3}^3 I(\bar\bu) +\Big(1+\|\rho_0\|_{L^{\infty}}\Big)\int_{\bbr^6}|v-\bu|^2f dx dv,
  \end{aligned}
\end{equation}
for some positive constants $C$ and $K_1$. In the above estimate, we have used the Sobolev inequality
\begin{equation}\label{eq-w-sobo}
  \|\bw\|_{L^9}^3 \le C I(\bw) \quad\text{for some constant $C>0$,}
\end{equation}
similarly for $\|\bar\bu\|_{L^9}^3$, and the fact that
\[
 \begin{aligned}
  \frac12\|\bh\|_{L^{\frac32}}^2\le \|\bh\|_{L^{1}}^2 +\|\bh\|_{L^{2}}^2
  \le\Big(1+\|\rho_0\|_{L^{\infty}}\Big)\int_{\bbr^6}|v-\bu|^2f dx dv.
  \end{aligned}
\]
Substituting \eqref{eq-w-ener-dt-rh} into \eqref{eq-w-ener-dt} and integrating the resulting equation over $[0,t]$, $t\in (0,T]$, we have by Lemma \ref{lm-cs-s-emapriori}(3) and \eqref{eq-baru-ener} that
\begin{equation}\label{eq-w-ener}
  \begin{aligned}
    &\sup_{0\le \tau\le t} \|\bw(\tau)\|_{L^3}^3  +2\int_0^t I(\bw) d\tau \\
    \le& K_1\int_0^t \|\bw\|_{L^3}^3 I(\bw) d\tau +C\sup_{0\le \tau\le t} \|\bar\bu(\tau)\|_{L^3}^3\int_0^t I(\bar\bu) d\tau\\
    &+\Big(1+\|\rho_0\|_{L^{\infty}}\Big)\int_0^t \int_{\bbr^6}|v-\bu|^2f dx dv d\tau\\
    \le&K_1\int_0^t \|\bw\|_{L^3}^3 I(\bw) d\tau + C\|\bu_0\|_{L^3}^6+(1+\|\rho_0\|_{L^{\infty}})E_0,
  \end{aligned}
\end{equation}
for some constants $K_1, C>0$, independent of the initial data. Since $\|\bw(0)\|_{L^3}^3=0$, there must exist $t^*\in (0, T]$ such that
\begin{equation}\label{eq-w-ener-smal}
 \|\bw(t)\|_{L^3}^3 <2\e_0 \quad\text{for $t\in [0, t^*)$,}
\end{equation}
where $\e_0$ satisfies $2K_1\e_0=1.$ If
\[
 C\|\bu_0\|_{L^3}^6+(1+\|\rho_0\|_{L^{\infty}})E_0\le\e_0,
\]
it follows from \eqref{eq-w-ener} and \eqref{eq-w-ener-smal} that
\begin{equation}\label{eq-w-ener-est}
  \begin{aligned}
    &\sup_{0\le \tau\le t^*} \|\bw(\tau)\|_{L^3}^3  +\int_0^{t^*} I(\bw) d\tau
    \le C\|\bu_0\|_{L^3}^6+(1+\|\rho_0\|_{L^{\infty}})E_0\le\e_0.
  \end{aligned}
\end{equation}
By continuity argument, we can show that $t^*=T$. Thus, it holds that
\begin{equation}\label{eq-w-ener-esteqv}
  \begin{aligned}
    &\sup_{0\le t\le T} \|\bw(t)\|_{L^3}^3  +\int_0^{T} I(\bw) dt
    \le C\|\bu_0\|_{L^3}^6+(1+\|\rho_0\|_{L^{\infty}})E_0.
  \end{aligned}
\end{equation}
Combining with \eqref{eq-baru-ener} and \eqref{eq-w-sobo}, we deduce that
\begin{equation}\label{eq-u-seri-cond}
  \begin{aligned}
    \|\bu\|_{L^3(0,T; L^9)}^3\le&C\Big(\|\bar\bu\|_{L^3(0,T; L^9)}^3+\|\bw\|_{L^3(0,T; L^9)}^3\Big)\\
    \le& C\int_0^{T} I(\bar\bu) dt+C\int_0^{T} I(\bw) dt\\
    \le& C\Big(\|\bu_0\|_{L^3}^3+ \|\bu_0\|_{L^3}^6+(1+\|\rho_0\|_{L^{\infty}})E_0 \Big),
  \end{aligned}
\end{equation}
for some constant $C>0$ independent of the initial data.
\end{proof}
In fact, $\bu\in {L^3(0,T; L^9)}$ satisfies the Serrin condition, which can be used to derive the estimate on $\sup_{0\le t\le T}\|\nabla \bu\|_{L^2}$.
\begin{lemma}\label{lm-cs-s-graduinftapriori}
Under the conditions in Theorem \ref{thm-exist}, if $(f, \bu)$ is a classical solution to \eqref{eq-cs-ns}-\eqref{eq-sys-inidata} in $[0, T]$, then it holds that
\[
 \sup_{0\le t\le T}\|\nabla \bu(t)\|_{L^2}^2 +\int_0^T \|\nabla^2\bu\|_{L^2}^2dt \le C\Big(\|\nabla \bu_0\|_{L^2}^2+\|\rho_0\|_{L^{\infty}}E_0\Big)\exp\left(C(\varepsilon_0+\varepsilon_0^{\frac12})\right),
\]
for some constant $C>0$ independent of the initial data.
\end{lemma}
\begin{proof}
Taking the dot product of $\eqref{eq-cs-ns}_2$ with $-\Delta \bu$, and integrating the resulting equation over $\bbr^3$, we obtain
\begin{equation}\label{eq-aprigu}
  \begin{aligned}
   &\frac{d}{dt}\|\nabla \bu \|_{L^2}^{2}+ K_2\|\nabla^2 \bu \|_{L^2}^{2}\\
   \le&\int_{\bbr^3}\bu \cdot\nabla \bu\cdot\Delta \bu dx -\int_{\bbr^3}\int_{\bbr^3} (v-\bu ) f dv\cdot\Delta \bu dx\\
   \le&\|\bu \|_{L^{9}} \|\nabla \bu \|_{L^{\frac{18}{7}}}\|\Delta \bu \|_{L^2}+\|\bh \|_{L^2}\|\Delta \bu \|_{L^2}\\
   \le&C \|\bu \|_{L^9} \|\nabla \bu \|_{L^2}^{\frac23} \|\nabla^2 \bu \|_{L^2}^{\frac43} +C\|\bh \|_{L^2}\|\nabla^2 \bu \|_{L^2}\\
   \le&\frac {K_2}{2} \|\nabla^2 \bu \|_{L^2}^2 +C\|\bu \|_{L^9}^{3}\|\nabla \bu \|_{L^2}^{2} +C\|\bh \|_{L^2}^2,
  \end{aligned}
\end{equation}
for some positive constants $K_2$ and $C$. Here we have used the Gagliardo--Nirenberg inequality
\begin{equation}\label{eq-gn-ulinfty}
 \|\nabla\bu \|_{L^{\frac{18}{7}}}  \le C \|\nabla\bu \|_{L^2}^{\frac23} \|\nabla^2 \bu \|_{L^2}^{\frac13} \quad \text{in $\bbr^3$,}
\end{equation}
and the elliptic estimate
\begin{equation}\label{eq-gtuelli}
 K_2\|\nabla^2 \bu \|_{L^2}^2 \le  \|\Delta \bu \|_{L^2}^2.
\end{equation}
Using Lemma \ref{lm-u-serr-cond} and the fact that
\[
 \int_0^T \|\bh\|_{L^2}^2 dt \le \|\rho_0\|_{L^{\infty}}E_0,
\]
we solve the above Gronwall inequality \eqref{eq-aprigu} to obtain
\begin{equation}\label{eq-gultwo}
 \sup_{0\le t\le T}\|\nabla \bu(t)\|_{L^2}^2+
 \int_0^T \|\nabla^2 \bu \|_{L^2}^{2} dt \le C\Big(\|\nabla \bu_0\|_{L^2}^2+\|\rho_0\|_{L^{\infty}}E_0\Big)\exp\left(C(\varepsilon_0+\varepsilon_0^{\frac12})\right),
\end{equation}
for some constant $C>0$ independent of the initial data.
\end{proof}
Having obtained the estimate on $\sup_{0\le t\le T}\|\nabla \bu(t)\|_{L^2}$, we then use Proposition \ref{prop-s} to derive the estimate on $\|\nabla^2\bu\|_{L^s(0,T: L^q)}$, which leads to the estimate on $\int_0^T \|\bu(t)\|_{W^{1,\infty}} dt$ by interpolation.
\begin{lemma}\label{lm-uwone-est}
Under the conditions in Theorem \ref{thm-exist}, if $(f, \bu)$ is a classical solution to \eqref{eq-cs-ns}-\eqref{eq-sys-inidata} in $[0, T]$, then it holds  that
\[
 \begin{aligned}
   &(1)\ R(t)\le R_0+C,\quad \text{for $t\in [0, T]$};\\
   &(2)\ \|\nabla^2\bu\|_{L^s(0,T; L^q)}\le C+CT^{\frac1s};\\
   &(3)\ \int_0^T \|\bu(t)\|_{W^{1,\infty}} dt \le C(1+T),
 \end{aligned}
\]
where $C:=C(R_0, E_0, \|\rho_0\|_{L^{\infty}}, \|\bu_0\|_{H^1},\|\bu_0\|_{L^3})$.
\end{lemma}
\begin{proof}
(1)\ Solving the characteristic equation \eqref{eq-charac} yields
\begin{equation}\label{eq-charac-v}
  \begin{aligned}
    V(t) =&v_0 \exp\left(-\int_0^t \Big[1+a(s,X(s))\Big]ds\right)\\
    &+\int_0^t \Big[\mathbf{b}(\tau,X(\tau))+\bu(\tau,X(\tau))\Big]\exp\left(-\int_{\tau}^t \Big[1+a(s,X(s))\Big]ds\right) d\tau.
  \end{aligned}
\end{equation}
Using Lemma \ref{lm-cs-s-emapriori}(3) and the Cauchy inequality, we have
\begin{equation}\label{eq-b-bound}
 \sup_{0\le t\le T}\|\mathbf{b}(t)\|_{L^{\infty}}\le  \left(\int_{\bbr^{6}} f |v|^2 dx dv \right)^{\frac12}\le E_0^{\frac12}.
\end{equation}
From \eqref{eq-charac-v}, we deduce that for $t\in [0, T]$,
\begin{equation}\label{eq-v-bound}
  \begin{aligned}
    R(t) \le&R_0+\int_0^t \|\bu(\tau)\|_{L^{\infty}}^2 d\tau\\
    &+\int_0^t \Big(1+\|\mathbf{b}(\tau)\|_{L^{\infty}}\Big)\exp\left(-\int_{\tau}^t \Big[1+a(s,X(s))\Big]ds\right) d\tau\\
    \le& R_0+C\int_0^t \|\nabla\bu(\tau)\|_{L^2} \|\nabla^2\bu(\tau)\|_{L^2} d\tau\\
    &+\bigg(1+\sup_{0\le \tau\le t}\|\mathbf{b}(\tau)\|_{L^{\infty}}\bigg) \int_0^t \Big[1+a(\tau,X(\tau))\Big]\exp\left(-\int_{\tau}^t \Big[1+a(s,X(s))\Big]ds\right) d\tau\\
    \le&R_0+1+ E_0^{\frac12}+ C \int_0^t \|\nabla\bu(\tau)\|_{L^2}^2 d\tau +C\int_0^t  \|\nabla^2\bu(\tau)\|_{L^2}^2 d\tau\\
    \le& R_0 +C(R_0, E_0, \|\rho_0\|_{L^{\infty}}, \|\bu_0\|_{H^1},\|\bu_0\|_{L^3}),
  \end{aligned}
\end{equation}
where we have used Lemma \ref{lm-cs-s-emapriori}(3),  Lemma \ref{lm-cs-s-graduinftapriori}, \eqref{eq-b-bound} and the Sobolev inequality
\[
 \|\bu(\tau)\|_{L^{\infty}}^2\le C\|\nabla\bu(\tau)\|_{L^2} \|\nabla^2\bu(\tau)\|_{L^2} \quad\text{in $\bbr^3$, for some constant $C>0$.}
\]
(2)\ Applying Proposition \ref{prop-s} to $\eqref{eq-cs-ns}_2-\eqref{eq-cs-ns}_3$, we obtain
\begin{equation}\label{eq-ulsq-est}
  \begin{aligned}
    \|\nabla^2\bu\|_{L^s(0,T;L^q)} \le C\Big(\|\bu_0\|_{H^1}+\|\bu\cdot\nabla\bu\|_{L^s(0,T;L^q)}+\|\bh\|_{L^s(0,T;L^q)}\Big).
  \end{aligned}
\end{equation}
We estimate $\|\bu\cdot\nabla\bu\|_{L^s(0,T;L^q)}$ and $\|\bh\|_{L^s(0,T;L^q)}$ as follows. It is easy to see that
\begin{equation}\label{eq-conv-lq}
 \begin{aligned}
  \|\bu\cdot\nabla\bu\|_{L^q}\le& \|\bu\|_{L^{\infty}} \|\nabla\bu \|_{L^q}\\
  \le& C\|\nabla\bu \|_{L^2}^{1-\theta_1} \|\nabla^2\bu \|_{L^q}^{\theta_1} \|\nabla\bu \|_{L^2}^{1-\theta_2} \|\nabla^2\bu \|_{L^q}^{\theta_2}\\
  \le&C\|\nabla\bu \|_{L^2}^{2-(\theta_1+\theta_2)} \|\nabla^2\bu \|_{L^q}^{\theta_1+\theta_2},
  \end{aligned}
\end{equation}
where we have use the following Gagliardo--Nirenberg inequalities in $\bbr^3$,
\[
  \|\bu\|_{L^{\infty}}\le C\|\nabla\bu \|_{L^2}^{1-\theta_1} \|\nabla^2\bu \|_{L^q}^{\theta_1}\quad \text{and}\quad \|\nabla\bu \|_{L^q}\le C\|\nabla\bu \|_{L^2}^{1-\theta_2} \|\nabla^2\bu \|_{L^q}^{\theta_2},
\]
for some constant $C>0$, with
\[
 -\frac{1-\theta_1}{2}+ \theta_1\left(2-\frac3q\right)=0, \quad -\frac{1-\theta_2}{2}+ \theta_2\left(2-\frac3q\right)=1-\frac3q.
\]
Thus,
\begin{equation}\label{eq-conv-lsq}
 \begin{aligned}
  C\|\bu\cdot\nabla\bu\|_{L^s(0, T;L^q)}\le C\left(\int_0^T \|\nabla\bu\|_{L^{2}}^{rs}dt\right)^{\frac1s}+ \frac12 \|\nabla^2\bu\|_{L^s(0,T;L^q)},
  \end{aligned}
\end{equation}
where $$r:=\frac{2-(\theta_1+\theta_2)}{1-(\theta_1+\theta_2)}=6-\frac6q.$$
Using Lemma \ref{lm-cs-s-emapriori}(2), Lemma \ref{lm-cs-s-graduinftapriori} and \eqref{eq-v-bound}, we have
\begin{equation}\label{eq-h-lq}
 \begin{aligned}
  \|\bh\|_{L^q}=&\left\|\int_{\bbr^3}(v-\bu)f dv\right\|_{L^q} \\
  \le& C\|\rho\|_{L^q}+ C\|\rho\|_{L^{\frac{6q}{6-q}}} \|\nabla\bu\|_{L^2}\\
  \le& C(R_0, E_0, \|\rho_0\|_{L^{\infty}}, \|\bu_0\|_{H^1},\|\bu_0\|_{L^3}).
  \end{aligned}
\end{equation}
Substituting \eqref{eq-conv-lsq} into \eqref{eq-ulsq-est}, we use Lemma \ref{lm-cs-s-graduinftapriori} and \eqref{eq-h-lq} to obtain
\begin{equation}\label{eq-ulsq-est-sim}
  \begin{aligned}
    \|\nabla^2\bu\|_{L^s(0,T;L^q)} \le& C\|\bu_0\|_{H^1}+C\left(\int_0^T \|\nabla\bu\|_{L^{2}}^{rs}dt\right)^{\frac1s}+C\|\bh\|_{L^s(0,T;L^q)}\\
    \le& C+CT^{\frac1s},
  \end{aligned}
\end{equation}
where $C:=C(R_0, E_0, \|\rho_0\|_{L^{\infty}}, \|\bu_0\|_{H^1},\|\bu_0\|_{L^3}).$

\noindent (3)\ From the following Gagliardo--Nirenberg inequality in $\bbr^3$,
\begin{equation}\label{eq-sobogu-linfty}
 \|\nabla\bu \|_{L^{\infty}}  \le C \|\nabla \bu \|_{L^2}^{1-\theta_3} \|\nabla^2 \bu \|_{L^q}^{\theta_3} \quad \text{with $-\frac{1-\theta_3}{2}+ \theta_3\left(2- \frac3q \right)=1$,}
\end{equation}
we deduce that
\begin{equation}\label{eq-grauone-infty}
  \begin{aligned}
   \int_0^T \|\nabla\bu \|_{L^{\infty}} dt\le &C \int_0^T \|\nabla \bu \|_{L^2}^{1-\theta_3} \|\nabla^2 \bu \|_{L^q}^{\theta_3} dt\\
   \le& C \int_0^T \|\nabla \bu \|_{L^2} dt+ C \int_0^T  \|\nabla^2 \bu \|_{L^q} dt\\
   \le& C\left(\int_0^T \|\nabla\bu\|_{L^{2}}^{2}dt\right)^{\frac12}T^{\frac12}+ C\left(\int_0^T \|\nabla^2\bu\|_{L^{q}}^{s}dt\right)^{\frac1s}T^{1-{\frac1s}}\\
    \le& C(1+T),
  \end{aligned}
\end{equation}
where $C:=C(R_0, E_0, \|\rho_0\|_{L^{\infty}}, \|\bu_0\|_{H^1},\|\bu_0\|_{L^3})$ and we have used Lemma \ref{lm-cs-s-emapriori}(3), along with \eqref{eq-ulsq-est-sim}.
Since
\[
 \int_0^T \|\bu \|_{L^{\infty}} dt\le \left(\int_0^T \|\bu \|_{L^{\infty}}^{2}dt\right)^{\frac12}T^{\frac12}\le  C(1+T),
\]
where $C:=C(R_0, E_0, \|\rho_0\|_{L^{\infty}}, \|\bu_0\|_{H^1},\|\bu_0\|_{L^3})$. Combining with \eqref{eq-grauone-infty}, we deduce that
\[
 \int_0^T \|\bu(t)\|_{W^{1,\infty}} dt \le C(1+T) \quad \text{for $C:=C(R_0, E_0, \|\rho_0\|_{L^{\infty}}, \|\bu_0\|_{H^1},\|\bu_0\|_{L^3})$.}
\]
This completes the proof.
\end{proof}
%%%%%%%%%%%%%%%%%%%%%%%%%%%%%%%%%%%%%%%%%%%%%%%%%%%%%%%%%%%%%%%%%%%%%%%%%%%%%%%%%%%%%%%%%%%%%%%%%%%%%%%%%%%
%
%                                Sect.5  Proof of the Theorems
%
%%%%%%%%%%%%%%%%%%%%%%%%%%%%%%%%%%%%%%%%%%%%%%%%%%%%%%%%%%%%%%%%%%%%%%%%%%%%%%%%%%%%%%%%%%%%%%%%%%%%%%%%%%%
\section{Proof of the Theorems}\label{sec-glob-exst}
\setcounter{equation}{0}
Combining the local existence result  with the a priori estimates on classical solutions to the coupled system, we prove global existence of strong solutions by continuity argument.
\vskip 3mm
\noindent \textit{Proof of Theorem \ref{thm-exist}.} From Proposition \ref{prop-loc-exist}, it follows that there exists some $T_0>0$ such that \eqref{eq-cs-ns}-\eqref{eq-sys-inidata} admits a unique strong solution in $[0, T_0]$. Take the supremum among all the $T_0$, and define the supremum $T^*$ as  the life span of strong solutions. Next, we demonstrate that $T^*=\infty$ by contradiction. Suppose not, i.e., $T^*<\infty$. We mollify the initial data by convolving with the standard mollifier, and then take limit to the obtained approximate classical solutions. Under conditions in Theorem \ref{thm-exist}, it follows from Proposition \ref{prop-kine-cs-wp} and Lemma \ref{lm-uwone-est} that the local strong solutions satisfy
\begin{equation}\label{eq-rf-est}
 \begin{aligned}
   &\|f(t)\|_{H_{\omega}^1}\le \|f_0\|_{H_{\omega}^1}\exp\Big(C(1+t)\Big) \quad \text{for $t\in [0, T^*)$};\\
   &R(t)\le R_0+C\quad \text{for $t\in [0, T^*)$},
 \end{aligned}
\end{equation}
where $C:=C(R_0, E_0, \|\rho_0\|_{L^{\infty}}, \|\bu_0\|_{H^1},\|\bu_0\|_{L^3})$.
Combining Lemma \ref{lm-cs-s-emapriori} and \ref{lm-cs-s-graduinftapriori}, we deduce that for $t\in [0, T^*)$,
\[
 \|\bu(t)\|_{H^{1}}^2\le E_0 + C\Big(\|\nabla \bu_0\|_{L^2}^2+\|\rho_0\|_{L^{\infty}}E_0\Big)\exp\left(C(\varepsilon_0+\varepsilon_0^{\frac12})\right).
\]
Thus, we have
\begin{equation}\label{eq-uhone-est}
 \|\bu(t)\|_{H^{1}}\le C\bigg(\|\nabla \bu_0\|_{L^2}+\Big(1+\|\rho_0\|_{L^{\infty}}^{\frac12}\Big)E_0^{\frac12}\bigg)\exp\left(C(\varepsilon_0+\varepsilon_0^{\frac12})\right)
\end{equation}
for $t\in [0, T^*)$ and some constant $C>0$. In terms of continuity of $f(t)$ and $\bu(t)$, we can define
\begin{equation}\label{eq-rfulim}
 \begin{aligned}
   &f(T^*):= \lim_{t \to T^*-}f(t) \quad \text{in $H_{\omega}^1(\bbr^3\times \bbr^3)$};\\
   &\bu(T^*):=\lim_{t \to T^*-}\bu(t) \quad \text{in $H^1(\bbr^3)$}.
 \end{aligned}
\end{equation}
From $\eqref{eq-rf-est}_2$, we know
\[
 R(T^*)\le R_0+C(R_0, E_0, \|\rho_0\|_{L^{\infty}}, \|\bu_0\|_{H^1},\|\bu_0\|_{L^3}).
\]
Thus, we can take $\Big(f(T^*), \bu(T^*)\Big)$ as an initial datum, and use Proposition \ref{prop-loc-exist} to continue the local strong solution past $T^*$, which contradicts the definition of $T^*$. Therefore, $T^*=\infty$, i.e., the Cauchy problem \eqref{eq-cs-ns}-\eqref{eq-sys-inidata} admits a unique global-in-time strong solution.

Using the equations $\eqref{eq-cs-ns}_1$ and $\eqref{eq-cs-ns}_2$, we deduce that
\begin{equation}\label{eq-ali-uv-dt}
 \begin{aligned}
   &\frac{d}{dt}\int_{\bbr^6} |v-\bu|^2 f dxdv\\
   =&\int_{\bbr^6} |v-\bu|^2 f_t dxdv -2\int_{\bbr^6} (v-\bu)\cdot\bu_t f dxdv\\
   =&-\int_{\bbr^6}\Big[\bu \cdot \nabla_{x} f+ \nabla_{v} \cdot (L[f]f+(\bu-v)f)\Big] |v-\bu|^2 dxdv\\
   &-2\int_{\bbr^6} \bigg(-\bu\cdot\nabla\bu-\nabla P+\Delta\bu+\int_{\bbr^3} (v-\bu) f dv\bigg)\cdot(v-\bu) f dxdv\\
   =&\int_{\bbr^6} 2L[f]\cdot (v-\bu) f dxdv- 2\int_{\bbr^6} |v-\bu|^2 f dxdv\\
   &+\int_{\bbr^6} 2(v-\bu)\cdot\nabla P f dxdv-2\int_{\bbr^6}\Delta\bu\cdot (v-\bu) f dxdv -2\int_{\bbr^3} |\bh|^2  dx\\
   \le& \int_{\bbr^6}\int_{\bbr^6}\varphi(|x-y|)f(t,y,v^*)f(t,x,v)\Big(|v^*-v|^2 + |v-\bu|^2\Big)dy dv^* dx dv \\
   &- 2\int_{\bbr^6} |v-\bu|^2 f dxdv-2\int_{\bbr^3} \bh\cdot\nabla P dx-2\int_{\bbr^3}\Delta\bu\cdot \bh  dx -2\int_{\bbr^3} |\bh|^2  dx\\
   \le&- \int_{\bbr^6} |v-\bu|^2 f dxdv +\|\nabla P\|_{L^2}^2 +\|\Delta\bu\|_{L^2}^2\\
   & +\int_{\bbr^6}\int_{\bbr^6}\varphi(|x-y|)f(t,y,v^*)f(t,x,v)|v^*-v|^2dy dv^* dx dv.
 \end{aligned}
\end{equation}
It follows from Proposition \ref{prop-s} that
\begin{equation}\label{eq-dudp-ltwo}
 \begin{aligned}
   &\int_0^T \|\nabla P\|_{L^2}^2 +\|\Delta\bu\|_{L^2}^2 dt\\ \le& C\left(\|\bu_0\|_{H^1}^2 +\int_0^T \|\bu\cdot\nabla\bu\|_{L^2}^2 dt +\int_0^T \|\bh\|_{L^2}^2 dt\right)\\
   \le&C\Bigg(\|\bu_0\|_{H^1}^2 +\sup_{0\le t\le T} \|\nabla\bu\|_{L^2}^2 \int_0^T \|\bu\|_{L^{\infty}}^2 dt +\|\rho_0\|_{L^{\infty}} \int_0^T \int_{\bbr^6} |v-\bu|^2 f dxdv dt\Bigg)\\
   \le&C(R_0, E_0, \|\rho_0\|_{L^{\infty}}, \|\bu_0\|_{H^1},\|\bu_0\|_{L^3}),
 \end{aligned}
\end{equation}
for $T\in (0,\infty)$, where we have used Lemma \ref{lm-cs-s-emapriori}(3) and Lemma \ref{lm-cs-s-graduinftapriori}.
Define
\[
 D(t):= \int_{\bbr^6}\int_{\bbr^6}\varphi(|x-y|)f(t,y,v^*)f(t,x,v)|v^*-v|^2dy dv^* dx dv
   +\|\nabla P\|_{L^2}^2 +\|\Delta\bu\|_{L^2}^2.
\]
Combining Lemma \ref{lm-cs-s-emapriori}(3) with \eqref{eq-dudp-ltwo}, we infer that for all $T\in (0, \infty)$,
\begin{equation}\label{eq-dissi-est}
 \begin{aligned}
   \int_0^T D(t) dt
   \le C(R_0, E_0, \|\rho_0\|_{L^{\infty}}, \|\bu_0\|_{H^1},\|\bu_0\|_{L^3}).
 \end{aligned}
\end{equation}
Solving the Gronwall inequality \eqref{eq-ali-uv-dt} yields
\begin{equation}\label{eq-ali-uv}
 \begin{aligned}
   \int_{\bbr^6} |v-\bu|^2 f dxdv\le& \int_{\bbr^6} |v-\bu_0|^2 f_0 dxdv e^{-t} +\int_0^t e^{-(t-s)} D(s) ds\\
   \le&C  e^{-t} + e^{-\frac{t}{2}} \int_0^{\frac{t}{2}} D(s) ds+ \int_{\frac{t}{2}}^t D(s) ds.
 \end{aligned}
\end{equation}
It follows from \eqref{eq-dissi-est} and \eqref{eq-ali-uv} that
\[
 \lim_{t \to \infty}\int_{\bbr^6} |v-\bu|^2 f dxdv =0.
\]
Thus the proof of Theorem \ref{thm-exist} is completed. $\hfill \square$

If $\bu_0\in L^1$ in addition, then the quantitative decay rate of $E(t)$ can be achieved by means of the Fourier splitting method.
\vskip 3mm
\noindent \textit{Proof of Theorem \ref{thm-beha}.} Adding \eqref{eq-cs-ener-dt} to \eqref{eq-s-ener-dt} yields
\begin{equation}\label{eq-elener-dt}
 \begin{aligned}
   \frac{d}{dt}E(t) +\|\nabla\bu\|_{L^2}^2+\int_{\bbr^6}|v-\bu|^2 f dxdv \le 0.
 \end{aligned}
\end{equation}
Since
\begin{equation}\label{eq-du-fsplit}
 \begin{aligned}
   \|\nabla\bu\|_{L^2}^2=& \int_{\bbr^3} |\hat\bu(\xi)|^2 |\xi|^2 d\xi\\
   \ge&\int_{|\xi|\ge c(t+c^2)^{-\frac12}} |\hat\bu(t,\xi)|^2 |\xi|^2 d\xi\\
   \ge& \frac{c^2}{t+c^2} \int_{|\xi|\ge c(t+c^2)^{-\frac12}} |\hat\bu(t,\xi)|^2d\xi\\
   =& \frac{c^2}{t+c^2}  \int_{\bbr^3} |\bu(t,x)|^2 dx- \frac{c^2}{t+c^2} \int_{|\xi|\le c(t+c^2)^{-\frac12}} |\hat\bu(t,\xi)|^2d\xi,
 \end{aligned}
\end{equation}
for some constant $c>0$ to be determined later. Here we have used the Plancherel Theorem
\[
 \|\hat\bu(t)\|_{L^2}^2=\|\bu(t)\|_{L^2}^2.
\]
Substituting \eqref{eq-du-fsplit} into \eqref{eq-elener-dt} results in
\begin{equation}\label{eq-elener-fs-dt}
 \begin{aligned}
   \frac{d}{dt}E(t) +\frac{c^2}{t+c^2}\left(\|\bu\|_{L^2}^2+\int_{\bbr^6}|v-\bu|^2 f dxdv \right)\le \frac{c^2}{t+c^2} \int_{|\xi|\le c(t+c^2)^{-\frac12}} |\hat\bu(t,\xi)|^2d\xi.
 \end{aligned}
\end{equation}
We observe that
\begin{equation}\label{eq-ener-eqv}
 \begin{aligned}
   &\|\bu\|_{L^2}^2+\int_{\bbr^6}|v|^2 f dxdv\\
   \le& \|\bu\|_{L^2}^2+ 2\int_{\bbr^6}|v-\bu|^2 f dxdv +2\int_{\bbr^6}|\bu|^2 f dxdv\\
   \le& \|\bu\|_{L^2}^2+ 2\int_{\bbr^6}|v-\bu|^2 f dxdv +2\|\rho_0\|_{L^{\infty}} \|\bu\|_{L^2}^2\\
   \le&2 (1+\|\rho_0\|_{L^{\infty}}) \left(\|\bu\|_{L^2}^2+\int_{\bbr^6}|v-\bu|^2 f dxdv \right),
 \end{aligned}
\end{equation}
where we have used Lemma \ref{lm-cs-s-emapriori}(2).
Substituting \eqref{eq-ener-eqv} into \eqref{eq-elener-fs-dt} gives
\begin{equation}\label{eq-elener-fsn-dt}
 \begin{aligned}
   \frac{d}{dt}E(t) +\frac{c^2}{(t+c^2)(1+\|\rho_0\|_{L^{\infty}})}E(t) \le \frac{c^2}{t+c^2}\int_{|\xi|\le c(t+c^2)^{-\frac12}} |\hat\bu(t,\xi)|^2d\xi.
 \end{aligned}
\end{equation}
Take $c^2=3(1+\|\rho_0\|_{L^{\infty}})$. Then \eqref{eq-elener-fsn-dt} becomes
\begin{equation}\label{eq-elener-fsngro-dt}
 \begin{aligned}
   \frac{d}{dt}E(t) +\frac{3}{t+c^2}E(t) \le \frac{c^2}{t+c^2}\int_{|\xi|\le c(t+c^2)^{-\frac12}} |\hat\bu(t,\xi)|^2d\xi.
 \end{aligned}
\end{equation}
Applying the Fourier transform to $\eqref{eq-cs-ns}_2$ leads to
\begin{equation}\label{eq-flueq-ftrans}
 \begin{aligned}
   \hat\bu_t+\widehat{\bu\cdot\nabla\bu} +\widehat{\nabla P}=-|\xi|^2\hat\bu +\hat\bh.
 \end{aligned}
\end{equation}
It follows from \eqref{eq-flueq-ftrans} that
\begin{equation}\label{eq-flueq-ftrans-uest}
 \begin{aligned}
   |\hat\bu(t,\xi)|\le&  |\hat\bu_0(\xi)| e^{-t|\xi|^2}\\
   &+\int_0^t e^{-(t-s)|\xi|^2}
   \Big(|\widehat{\bu\cdot\nabla\bu}|(s,\xi) +|\widehat{\nabla P}|(s,\xi)+|\hat\bh|(s,\xi) \Big) ds.
 \end{aligned}
\end{equation}
Using properties of Fourier's transform and Young's inequality for convolutions, we have
\begin{equation}\label{eq-flueq-ftrans-convest}
 \begin{aligned}
   |\widehat{\bu\cdot\nabla\bu}|(s,\xi)= |\Widehat{\nabla\cdot(\bu\otimes\bu)}|(s,\xi)\le|\xi|\cdot |\widehat{\bu\otimes\bu}|(s,\xi)\le C |\xi| \cdot \|\bu(s)\|_{L^2}^2,
 \end{aligned}
\end{equation}
for some constant $C>0$. Here $\bu\cdot\nabla\bu=\nabla\cdot(\bu\otimes\bu)$ is due to the fact that $\nabla\cdot\bu=0$.
Taking divergence to $\eqref{eq-cs-ns}_2$ yields
\begin{equation}\label{eq-lapl-p}
 \begin{aligned}
  \Delta P=-\nabla\cdot(\bu\cdot\nabla\bu)+\nabla\cdot\bh.
 \end{aligned}
\end{equation}
From \eqref{eq-flueq-ftrans-convest} and \eqref{eq-lapl-p}, we deduce that
\begin{equation}\label{eq-dp-ftrans-est}
 \begin{aligned}
   |\widehat{\nabla P}|(s,\xi)\le& C\Big(|\widehat{\bu\cdot\nabla\bu}|(s,\xi)+ |\hat{\bh}|(s,\xi)\Big)\\
   \le& C |\xi|\cdot \|\bu(s)\|_{L^2}^2+ C|\hat{\bh}|(s,\xi),
 \end{aligned}
\end{equation}
for some constant $C>0$. Substituting \eqref{eq-flueq-ftrans-convest} and \eqref{eq-dp-ftrans-est} into \eqref{eq-flueq-ftrans-uest}, we obtain
\begin{equation}\label{eq-flueq-ftrans-uestsim}
 \begin{aligned}
   |\hat\bu(t,\xi)|\le&  |\hat\bu_0(\xi)| e^{-t|\xi|^2}
   +C \int_0^t e^{-(t-s)|\xi|^2}
   \Big( |\xi|\cdot \|\bu(s)\|_{L^2}^2+ |\hat{\bh}|(s,\xi) \Big) ds\\
   \le&\|\bu_0\|_{L^1} +C\int_0^t |\xi|\cdot \|\bu(s)\|_{L^2}^2 ds +C\int_0^t \|\bh(s)\|_{L^1}ds,
 \end{aligned}
\end{equation}
for some constant $C>0$, where we have used the Hausdorff--Young inequality
\[
 \|\hat\bu_0\|_{L^{\infty}}\le \|\bu_0\|_{L^1},
\]
similarly for $|\hat{\bh}|(s,\xi)$. From Lemma \ref{lm-cs-s-emapriori}(3), it is easy to see that
\[
 E(t)\le E_0 \quad \text{for all $t\in [0, \infty)$}.
\]
Thus, there must exist some constant $K>1$, sufficiently large such that
\begin{equation}\label{eq-ener-bootassum}
 \begin{aligned}
   E(t)<K\Big(1+t\Big)^{-\frac98} \quad \text{for $t\in [0, T]$}.
 \end{aligned}
\end{equation}
Denote by $T_m$ the supremum among all the $T$ in \eqref{eq-ener-bootassum}. We prove by contradiction that $T_m=\infty$. Suppose $T_m<\infty$. Then it holds that
\begin{equation}\label{eq-ener-bootassum-tm}
 \begin{aligned}
   E\Big(T_m\Big)=K\Big(1+T_m \Big)^{-\frac98}.
 \end{aligned}
\end{equation}
It is easy to see that
\begin{equation}\label{eq-u-bootassum}
 \begin{aligned}
   \|\bu(t)\|_{L^2}^2\le 2E(t)< 2K\Big(1+t\Big)^{-\frac98} \quad \text{for $t\in [0, T_m)$}.
 \end{aligned}
\end{equation}
Multiplying \eqref{eq-elener-dt} by $\Big(1+t\Big)^{\frac{17}{16}}$ yields
\begin{equation}\label{eq-tweig-ener}
 \begin{aligned}
   \frac{d}{dt} \left(\Big(1+t\Big)^{\frac{17}{16}}E(t)\right) +\Big(1+t\Big)^{\frac{17}{16}}\int_{\bbr^6}|v-\bu|^2 f dxdv \le \frac{17}{16}\Big(1+t\Big)^{\frac{1}{16}}E(t).
 \end{aligned}
\end{equation}
Integrating \eqref{eq-tweig-ener} over $[0, T_m]$, we have
\begin{equation}\label{eq-tweig-coupest}
 \begin{aligned}
   \int_0^{T_m}\Big(1+t\Big)^{\frac{17}{16}}\int_{\bbr^6}|v-\bu|^2 f dxdvdt \le E_0+CK
 \end{aligned}
\end{equation}
for some some constant $C>0$, where we have used the fact that
\[
 E(t)<K\Big(1+t\Big)^{-\frac98} \quad \text{for $t\in [0, T_m)$}.
\]
From \eqref{eq-tweig-coupest}, we infer that
\begin{equation}\label{eq-coupest}
 \begin{aligned}
   C\int_0^{T_m}\|\bh(s)\|_{L^1}  ds \le& C\int_0^{T_m}\Big(1+s\Big)^{-\frac{17}{32}}\Big(1+s\Big)^{\frac{17}{32}}\|\bh(s)\|_{L^1} ds \\
   \le&C\left( \int_0^{T_m}\Big(1+s\Big)^{\frac{17}{16}}\int_{\bbr^6}|v-\bu|^2 f dxdvds\right)^{\frac12}\\
   \le&C(E_0) \Big(1+K\Big)^{\frac12}.
 \end{aligned}
\end{equation}
Substituting \eqref{eq-u-bootassum} and \eqref{eq-coupest} into \eqref{eq-flueq-ftrans-uestsim}, we deduce that
\begin{equation}\label{eq-ftrans-usquest}
 \begin{aligned}
   |\hat\bu(t,\xi)|^2
   \le C+CK +CK^2 |\xi|^2 \quad \text{for $t\in [0, T_m]$},
 \end{aligned}
\end{equation}
where $C:=C(\|\bu_0\|_{L^1}, E_0)$.
It follows from \eqref{eq-ftrans-usquest} that
\begin{equation}\label{eq-ftrans-usquestlow}
 \begin{aligned}
  \int_{|\xi|\le c(t+c^2)^{-\frac12}} |\hat\bu(t,\xi)|^2d\xi \le CK\Big(1+t\Big)^{-\frac{3}{2}}+ CK^2 \Big(1+t\Big)^{-\frac{5}{2}} \quad \text{for $t\in [0, T_m]$},
 \end{aligned}
\end{equation}
where $C:=C(\|\rho_0\|_{L^{\infty}}, \|\bu_0\|_{L^1}, E_0)$. Combining \eqref{eq-elener-fsngro-dt} with \eqref{eq-ftrans-usquestlow}, we have
\begin{equation}\label{eq-ener-decay}
 \begin{aligned}
  E(t) \le CK\Big(1+t\Big)^{-\frac{3}{2}}+ CK^2 \Big(1+t\Big)^{-\frac{5}{2}} \quad \text{for $t\in [0, T_m]$},
 \end{aligned}
\end{equation}
where $C:=C(\|\rho_0\|_{L^{\infty}}, \|\bu_0\|_{L^1}, E_0)$.
Since
\[
 E\Big(T_m\Big)=K\Big(1+T_m \Big)^{-\frac98}\le E_0,
\]
we have
\begin{equation}\label{eq-tm-est}
 1+T_m\ge \left(\frac{K}{E_0}\right)^{\frac89}.
\end{equation}
From \eqref{eq-ener-decay} and \eqref{eq-tm-est}, we infer that
\begin{equation}\label{eq-ener-decay-impr}
 \begin{aligned}
  E\Big(T_m\Big) \le& \bigg[C\Big(1+T_m \Big)^{-\frac38}+CK \Big(1+T_m \Big)^{-\frac{11}{8}} \bigg] K\Big(1+T_m \Big)^{-\frac98}\\
  \le& CK^{-\frac29} K\Big(1+T_m \Big)^{-\frac98},
 \end{aligned}
\end{equation}
for some constant $C:=C(\|\rho_0\|_{L^{\infty}}, \|\bu_0\|_{L^1}, E_0)$.
Take $K$ suitably large such that
\[
 CK^{-\frac29}=\frac12.
\]
Then we have $E\Big(T_m\Big) \le \frac K2 \Big(1+T_m \Big)^{-\frac98}$, which contradicts \eqref{eq-ener-bootassum-tm}. Therefore, we show that $T_m=\infty$, and it follows form \eqref{eq-ener-decay} that
\begin{equation}\label{eq-ener-decay-op}
 \begin{aligned}
  E(t) \le C\Big(1+t\Big)^{-\frac{3}{2}}
 \end{aligned}
\end{equation}
for some constant $C:=C(\|\rho_0\|_{L^{\infty}}, \|\bu_0\|_{L^1}, E_0)$. It is easy to deduce from \eqref{eq-ener-decay-op} that
\[
 \|\bu(t)\|_{L^2} \le C\Big(1+t\Big)^{-\frac{3}{4}}
\]
and
\begin{equation*}\label{eq-ener-decay-final}
 \begin{aligned}
 \int_{\bbr^6} |v-\bu|^2 f dxdv\le&2 \int_{\bbr^6} |v|^2 f dxdv+ 2\int_{\bbr^6} |\bu|^2 f dxdv\\
 \le& 2 \int_{\bbr^6} |v|^2 f dxdv+ 2\|\rho_0\|_{L^{\infty}} \|\bu(t)\|_{L^2}^2\\
 \le& CE(t)
 \le C\Big(1+t\Big)^{-\frac32},
 \end{aligned}
\end{equation*}
for some constant $C:=C(\|\rho_0\|_{L^{\infty}}, \|\bu_0\|_{L^1}, E_0)$, where we have used Lemma \ref{lm-cs-s-emapriori}(2). This completes the proof of Theorem \ref{thm-beha}. $\hfill \square$
\vskip 3mm
\noindent\textbf{Acknowledgements}. Chunyin Jin is supported by NSFC Grant No. 12001530.
%\bibliographystyle{plain}
%\bibliography{Jin-Ref}

\end{document}